\documentclass[11pt,a4paper]{article}
\usepackage{aligned-overset}
\usepackage{amsmath,amsthm}
\usepackage{authblk}
\usepackage[style=numeric-comp,maxbibnames=100,bibencoding=utf8,giveninits=true,backend=bibtex]{biblatex}
\usepackage{booktabs}
\usepackage{cancel}
\usepackage{comment}
\usepackage[hmargin={26mm,26mm},vmargin={30mm,35mm}]{geometry}
\usepackage{hyperref}
\usepackage{newtxtext}
\usepackage{newtxmath}
\usepackage{pgfplots,pgfplotstable}
\usepackage{amsfonts,latexsym,graphicx,verbatim,subcaption}
\usepackage{fancyvrb}
\usetikzlibrary{calc,external}

\bibliography{hho-friedrichs}

\newcommand{\email}[1]{\href{mailto:#1}{#1}}

\usepackage[normalem]{ulem}
\newcounter{corr}
\definecolor{violet}{rgb}{0.580,0.,0.827}
\newcommand{\corr}[3]{\typeout{Warning : a correction remains in page \thepage}
  \stepcounter{corr}
 	      {\color{blue}\ifmmode\text{\,\sout{\ensuremath{#1}}\,}\else\sout{#1}\fi}
              {\color{red}#2}
              {\color{violet} #3}
}

\theoremstyle{plain}
\newtheorem{theorem}{Theorem}
\newtheorem{proposition}[theorem]{Proposition}
\newtheorem{lemma}[theorem]{Lemma}
\newtheorem{corollary}[theorem]{Corollary}

\theoremstyle{remark}
\newtheorem{remark}[theorem]{Remark}

\theoremstyle{definition}
\newtheorem{assumption}{Assumption}

\newcommand{\Real}{\mathbb{R}}
\newcommand{\Complex}{\mathbb{C}}

\newcommand{\Id}[1]{\mathcal{I}_{#1}}

\newcommand{\trans}{\mathsf{T}}
\newcommand{\herm}{\mathsf{H}}

\DeclareMathOperator{\Ker}{Ker}

\newcommand{\st}{\;:\;}

\newcommand{\dirichlet}{\mathrm{D}}
\newcommand{\neumann}{\mathrm{N}}

\newcommand{\Th}{\mathcal{T}_h}
\newcommand{\Fh}{\mathcal{F}_h}
\newcommand{\Fhb}{\mathcal{F}_h^{\rm b}}
\newcommand{\Fhi}{\mathcal{F}_h^{\rm i}}
\newcommand{\FT}{\mathcal{F}_T}

\newcommand{\Poly}[1]{\mathcal{P}^{#1}}

\newcommand{\lproj}[2]{\pi^{#1}_{#2}}

\newcommand{\Uh}{\underline{U}_h^k}
\newcommand{\UT}{\underline{U}_T^k}

\newcommand{\Ih}{\underline{I}_h^k}

\newcommand{\norm}[2]{\|#2\|_{#1}}
\newcommand{\seminorm}[2]{|#2|_{#1}}

\newcommand{\tnorm}[2]{|\kern-0.25ex|\kern-0.25ex|#2|\kern-0.25ex|\kern-0.25ex|_{#1}}

\newcommand{\term}{\mathfrak{T}}

\newcommand{\Err}{\mathcal{E}_h}

\newcommand{\huline}[1]{\widehat{\underline{#1}}}

\newcommand{\Aref}[1][T]{\mathfrak{a}_{#1}}

\newcommand{\Cref}[1][T]{\mathfrak{c}_{#1}}
\newcommand{\tref}{\mathfrak{t}_T}

\newcommand{\Kmin}{\mathfrak{r}_\flat}
\newcommand{\Kmax}{\mathfrak{r}_\sharp}

\newcommand{\symm}{\mathrm{s}}

\newcommand{\Rm}{\mathfrak{Rm}}

\newcommand{\logLogSlopeTriangle}[5]{
  \pgfplotsextra
      {
        \pgfkeysgetvalue{/pgfplots/xmin}{\xmin}
        \pgfkeysgetvalue{/pgfplots/xmax}{\xmax}
        \pgfkeysgetvalue{/pgfplots/ymin}{\ymin}
        \pgfkeysgetvalue{/pgfplots/ymax}{\ymax}

        \pgfmathsetmacro{\xArel}{#1}
        \pgfmathsetmacro{\yArel}{#3}
        \pgfmathsetmacro{\xBrel}{#1-#2}
        \pgfmathsetmacro{\yBrel}{\yArel}
        \pgfmathsetmacro{\xCrel}{\xArel}

        \pgfmathsetmacro{\lnxB}{\xmin*(1-(#1-#2))+\xmax*(#1-#2)} 
        \pgfmathsetmacro{\lnxA}{\xmin*(1-#1)+\xmax*#1} 
        \pgfmathsetmacro{\lnyA}{\ymin*(1-#3)+\ymax*#3} 
        \pgfmathsetmacro{\lnyC}{\lnyA+#4*(\lnxA-\lnxB)}
        \pgfmathsetmacro{\yCrel}{\lnyC-\ymin)/(\ymax-\ymin)}

        \coordinate (A) at (rel axis cs:\xArel,\yArel);
        \coordinate (B) at (rel axis cs:\xBrel,\yBrel);
        \coordinate (C) at (rel axis cs:\xCrel,\yCrel);

        \draw[#5]   (A)-- node[pos=0.5,anchor=north] {\scriptsize{1}}
        (B)--
        (C)-- node[pos=0.,anchor=west] {\scriptsize{#4}} 
        cycle;
      }
}

\begin{document}

\title{Hybrid methods for Friedrichs systems with application to scalar and vector diffusion-advection-reaction}
\author[1]{Daniele A. Di Pietro}
\author[1]{Aurelio E. Spadotto}
\affil[1]{%
  Universit\'e de Montpellier,
  place Eug\`ene Bataillon, 34090, Montpellier, France \\
  \email{daniele.di-pietro@umontpellier.fr},
  \email{aurelio-edoardo.spadotto@umontpellier.fr}
}
\maketitle

\begin{abstract}
  In this work we study arbitrary-order hybrid discretizations of Friedrichs systems.
  Friedrichs systems provide a framework that goes beyond the standard classification of partial differential equations into hyperbolic or elliptic, and are thus particularly suited for problems that include both diffusive and advective terms.
  The family of numerical schemes proposed in this work hinge on hybrid spaces with unknowns located at elements and faces.
  They support general meshes, are locally conservative and, compared with traditional Discontinuous Galerkin discretizations, lead to smaller algebraic systems once static condensation has been applied.
  We carry out a complete stability and convergence analysis, which appears to be the first of its kind.
  The performance of the method is illustrated on scalar and vector three-dimensional diffusion-advection-reaction problems.%
  \smallskip\\
  \textbf{MSC2020:} %
  65N30, 
  65N08, 
  35F15  
  \smallskip\\
  \textbf{Key words:} Friedrichs systems, Hybrid High-Order methods, hybridizable discontinuous Galerkin methods, polytopal methods, diffusion-advection-reaction
\end{abstract}



\section{Introduction}

In this work we carry out what appears to be the first complete theoretical study of arbitrary-order hybrid discretizations of Friedrichs systems and showcase their performance on scalar and vector diffusion-advection-reaction problems.
The proposed family of methods support general meshes, are locally conservative, and, compared with classical Discontinuous Galerkin (DG) methods, lead to smaller algebraic problems after static condensation of the element unknowns.
\smallskip

Relevant partial differential equations (PDEs) can exhibit elliptic or hyperbolic behaviours depending on the value of the physical parameters.
A prototypical example is provided by diffusion-advection-reaction and its vanishing diffusion limit.
Weak formulations and robust numerical discretizations for such PDEs are often not straightforward to derive.
A significant advance in this respect was made by Friedrichs, who introduced in~\cite{Friedrichs:58} a framework to treat in a unified way elliptic and hyperbolic problems.
The key idea consists in recasting the PDE as a system of first-order equations satisfying symmetry and positivity conditions, which is possible for a variety of models in continuum mechanics and electromagnetism; see, e.g.,~\cite[Section~56.2]{Ern.Guermond:21*2}.
Deriving weak formulations of Friedrichs systems involves, however, some technicalities.
In~\cite{Ern.Guermond:06}, the authors introduce a weak formulation in graph spaces with boundary conditions enforced through (positive) boundary operators that do not necessarily admit a trace representation.
Such formulation is amenable to the analysis in the spirit of the Banach--Ne\v{c}as--Babu\v{s}ka Lemma (see~\cite[Chapter~56]{Ern.Guermond:21*2}), and will be the starting point of the present work.
\smallskip

Various approaches have been considered in the literature for the numerical approximation of Friedrichs systems since the late 1990s.
A first analysis of finite volume methods for unsteady Friedrichs systems can be found in~\cite{Vila.Villedieu:97}, where the authors show convergence in $h^{\frac12}$ (with $h$ denoting the meshsize) under a hyperbolic CFL condition; see also~\cite{Jovanovic.Rohde:05} for an extension of these results to weak solutions and a posteriori error estimates.
Discontinuous Galerkin methods were considered in~\cite{Ern.Guermond:06}, using the techniques developed for the imposition of boundary conditions as a paradigm for the weak enforcement of interface conditions.
The same authors also discussed enhanced DG methods for Friedrichs systems with two- or multi-field structure in~\cite{Ern.Guermond:06*1,Ern.Guermond:08}.
A related contribution is~\cite{Di-Pietro.Ern.ea:08} which, taking inspiration from the theory of two-fields Friedrichs systems, introduced a DG method for variable diffusion with advection that is fully robust for locally vanishing diffusion.
Discontinuous Galerkin methods for Friedrichs systems in graph spaces were also considered in~\cite{Jensen:06}, while %
discontinuous Petrov--Galerkin methods were proposed and analyzed in~\cite{Bui-Thanh.Demkowicz.ea:13}.
More recently, $hp$-adaptive hybridizable DG schemes for Friedrichs systems with one- and two-fields structure have been proposed and numerically demonstrated in~\cite{Chen.Kang.ea:24}, which also contains a coercivity study.
A side result of our analysis are inf-sup stability and convergence results for the $h$-version of the one-field scheme.
Space-time DG methods were considered in~\cite{Corallo.Dorfler.ea:23}, while~\cite{Burman.Ern:07} focuses on stabilized continuous finite element approximations.
Finally, the series of papers~\cite{Buet.Despres.ea:15,Despres.Buet:16,Morel.Buet.ea:18} addresses various aspects related to the development of asymptotic-preserving and well-balanced schemes for unsteady Friedrichs systems.
\smallskip

In this work we carry out what appears to be the first comprehensive analysis of hybrid approximations of Friedrichs systems on general polytopal meshes.
Hybrid methods are based on unknowns that are local polynomials on mesh elements and faces.
Among the most prominent examples, we cite here hybridizable DG~\cite{Cockburn.Gopalakrishnan.ea:09} methods, built around the notion of numerical flux, and Hybrid High-Order methods~\cite{Di-Pietro.Ern:15,Di-Pietro.Ern.ea:14}, based on the notion of local reconstructions; see~\cite{Cockburn.Di-Pietro.ea:16} and also~\cite[Section~5.1.6]{Di-Pietro.Droniou:20} for a study of the relations between these two methods.
The scheme we propose and study here can be regarded as a generalization of the face-element upwind discretization of advective terms considered in~\cite{Di-Pietro.Droniou.ea:15}.
Our analysis hinges on inf-sup stability, which we prove by taking a suitable projection of the directional derivative as a test function.
Denoting by $k \ge 0$ the polynomial degree of the hybrid space, this stability result is instrumental in proving error estimates of order $h^{k+\frac12}$ for the error, which includes an $L^2$-like norm of the solution components and of the (discrete) directional derivative.
Closely following~\cite{Ern.Guermond:21*2}, we carefully track the dependence on the model parameters, which makes it possible to uncover improved pre-asymptotic convergence in $h^{k+1}$ when reaction dominates.
The performance of the proposed method is showcased on various diffusion-advection-reaction models.
We consider, in particular, the vector case corresponding to magnetic diffusion-advection-reaction, which we preliminarily show to fit within the Friedrichs framework.
Stabilized methods for this problem were developed in~\cite{Wang.Wu:24,Wang.Wu:24*1,Luo.Wang.ea:26}.
A scheme supporting general polytopal meshes for magnetic advection-reaction has also been proposed in~\cite{Cantin.Ern:17}.
The performance of the method is numerically demonstrated on three-dimensional tests including, in particular, a physical benchmark inspired by~\cite{Perry.Jones:78} and describing the expulsion of the magnetic field from a rotating cilinder.
\smallskip

The rest of this work is organized as follows.
The continuous and discrete settings are respectively discussed in Sections~\ref{sec:continuous.setting} and~\ref{sec:discrete.setting}.
Section~\ref{sec:scheme} contains the statement of the discrete problem.
Complete stability and error analyses are carried out in Section~\ref{sec:analysis}.
Finally, Section~\ref{sec:numerical.tests} contains the numerical validation of the method.


\section{Continuous setting}\label{sec:continuous.setting}

\subsection{Friedrichs systems}

We briefly recall the notion of Friedrichs system.
The setting and notation are closely inspired by~\cite[Chapter~1]{Ern.Guermond:21*2}.
Denote by $\Omega$ a bounded polytopal domain of $\Real^d$, $d \ge 1$, with boundary $\partial \Omega$.
Given an integer $m > 0$, let $\mathcal{K} \in L^\infty(\Omega;\Complex^{m\times m})$
and let $\mathcal{A} \coloneqq \{\mathcal{A}^i\}_{1\le i \le d}$ be a family of $(m\times m)$-matrix valued fields such that $\mathcal{A}^i \in L^\infty(\Omega;\Complex^{m\times m}) \cap C^{0,\frac12}(\Omega; \Complex^{m\times m})$ for all $1\le i\le d$ and $\nabla \cdot \mathcal{A} \coloneqq \sum_{i=1}^d \partial_i \mathcal{A}^i \in L^\infty(\Omega;\Complex^{m\times m})$, where $\partial_i$ denotes the partial derivative with respect to the $i$th variable.
We additionally assume that the fields $\mathcal{A}^i$ are Hermitian, i.e., $\mathcal{A}^i = \left( \mathcal{A}^i\right)^\herm$ for all $1 \le i \le d$, and that there exists a real number $\Kmin > 0$ such that, almost everywhere in $\Omega$,
\begin{equation}\label{eq:lower.bound.K+KH-X}
  \mathcal{K} + \mathcal{K}^\herm - \nabla \cdot \mathcal{A} \ge 2 \Kmin \Id{m},
\end{equation}
where $\mathcal{I}_m \in \Real^{m \times m}$ denotes the identity matrix.
We define the following differential operators on $C^1(\overline{\Omega};\Complex^m)$:
\begin{equation}\label{eq:A1.A.tA}
  A_1 \coloneqq \sum_{i = 1}^d \mathcal{A}^i \partial_i,\qquad
  A \coloneqq \mathcal{K} + A_1,\qquad
  \tilde{A} \coloneqq
  \mathcal{K}^\herm - \nabla \cdot \mathcal{A} - A_1
  = \mathcal{K} + \mathcal{K}^\herm - \nabla \cdot \mathcal{A} - A,
\end{equation}
as well as the Hermitian field $\mathcal{N} \in L^\infty(\partial \Omega; \Complex^{m\times m})$ such that
\begin{equation}\label{eq:NF}
  \mathcal{N} \coloneqq \sum_{i = 1}^d n_i \mathcal{A}^i_{|\partial \Omega},
\end{equation}
where $n$ denotes the outward unit normal field on $\partial \Omega$.
In order to enforce boundary conditions, we will need another non-negative field $\mathcal{M} \in L^\infty(\partial \Omega; \Complex^{m\times m})$ such that $\Ker(\mathcal{M} - \mathcal{N}) + \Ker(\mathcal{M} + \mathcal{N}) = \Complex^m$.
Given $f \in L^2(\Omega;\Complex^m)$, we consider the problem that consists in finding $u : \Omega \to \Complex^m$ such that
\begin{equation}\label{eq:strong}
  \begin{alignedat}{2}
    A u &= f &\qquad& \text{in $\Omega$},
    \\
    (\mathcal{M} - \mathcal{N}) u &= 0 &\qquad& \text{on $\partial \Omega$}.
  \end{alignedat}
\end{equation}

The weak formulation of problem~\eqref{eq:strong} involves some technicalities and is discussed in detail in~\cite[Section~56.3]{Ern.Guermond:21*2}.
Denote by $V \coloneqq \left\{ v \in L^2(\Omega;\Complex^m) \st A_1 v \in L^2(\Omega;\Complex^m)\right\}$ the $L^2$-graph space of $A_1$.
Both $A$ and $\tilde{A}$ can be extended to operators in $\mathcal{L}(V;L^2(\Omega;\Complex^{m\times m}))$.
To circumvent the problem of traces for functions in $V$, we introduce the boundary operator $N \in \mathcal{L}(V;V')$ (with $V'$ dual space of $V$) such that $\langle N w, v\rangle_{V',V} \coloneqq (A w, v)_\Omega - (w, \tilde{A} v)_\Omega$ for all $(w,v) \in V^2$, where $(a,b)_\Omega \coloneqq \int_\Omega (a, b)_{\Complex^m}$ with $(\alpha,\beta)_{\Complex^m} \coloneqq \beta^\herm \alpha$ for all $\alpha,\,\beta \in \Complex^m$ (interpreted as column vectors).
It can be proved that $N$ is a boundary operator.
The role of $\mathcal{M}$ is played by the monotone operator $M \in \mathcal{L}(V;V')$ such that $\Ker(N - M) + \Ker(N + M) = V$.
Setting $V_0 \coloneqq \Ker(M - N)$, the weak formulation consists in finding $u \in V_0$ such that
\begin{equation}\label{eq:weak}
  (A u, v)_\Omega = (f, v)_\Omega \qquad \forall v \in L^2(\Omega;\Complex^m).
\end{equation}

Throughout the rest of the paper, we will bridge the strong and weak formulation by assuming that, for all $(w,v) \in H^s(\Omega;\Complex^m)$, $s > \frac12$, with $(a,b)_{\partial \Omega} \coloneqq \int_{\partial \Omega} (a, b)_{\Complex^m} $ usual inner product of $L^2(\partial \Omega; \Complex^m)$,
\[
\langle M w, v\rangle_{V',V} = (\mathcal{M} w, v)_{\partial \Omega},\qquad
\langle N w, v\rangle_{V',V} = (\mathcal{N} w, v)_{\partial \Omega}.
\]

\subsection{Examples}

We briefly discuss two examples that will be used in the numerical experiments of Section~\ref{sec:numerical.tests}.

\subsubsection{Scalar diffusion-advection-reaction}\label{subsubsec:scalar-dar}

Let $\partial\Omega=\Gamma_\dirichlet\cup\Gamma_\neumann$ with $\Gamma_\dirichlet\cap\Gamma_\neumann=\emptyset$ and $\Gamma_\dirichlet$ non-empty.
Consider a uniformly elliptic matrix-valued field $\kappa \in L^\infty(\Omega;\Real^{d\times d})$ and let $\mu \in L^\infty(\Omega;\Real)$.
Let $\beta \in L^\infty(\Omega;\Real^d) \cap C^{0,\frac12}(\Omega;\Real^d)$ be such that $\nabla\cdot\beta\in L^\infty(\Omega;\Real^d)$ and
assume that there exists $\Kmin>0$ such that $\mu -\frac12 \nabla\cdot\beta \ge \Kmin$ almost everywhere in $\Omega$.
Given $f_p \in L^2(\Omega;\Real^d)$, we consider the problem of finding the flux $\sigma:\Omega \to \Real^d$ and the scalar potential $p : \Omega \to \Real$ such that
\begin{equation}\label{eq:scalar_DAR}
  \begin{aligned}
    \kappa^{-1} \sigma + \nabla p &= 0 &\qquad& \text{in $\Omega$},
    \\
    \nabla \cdot \sigma + \beta\cdot\nabla p + \mu p &= f_p &\qquad& \text{in $\Omega$},
    \\
    p &= 0 &\qquad& \text{on $\Gamma_\dirichlet$},
    \\
    (\sigma + \beta p) \cdot n &= 0 &\qquad& \text{on $\Gamma_\neumann$}.
  \end{aligned}
\end{equation}
It is known that this problem can be recast into the form~\eqref{eq:strong} with $m = d + 1$ by setting $f \coloneqq
\begin{bmatrix}
  \mathcal{O}_d\\
  f_p
\end{bmatrix}$ and letting
\[
\mathcal{K} \coloneqq \begin{bmatrix}
  \kappa^{-1} & \mathcal{O}_{d} \\
  \mathcal{O}_{d}^\trans & \mu
\end{bmatrix},\qquad
\mathcal{A}^i \coloneqq \begin{bmatrix}
  \mathcal{O}_{d,d} & e_i \\
  e_i^\trans & \beta_i
\end{bmatrix},
\quad
\mathcal{N} \coloneqq \begin{bmatrix}
  \mathcal{O}_{d,d} & n \\
  n^\trans & \beta\cdot n
\end{bmatrix},
\quad
\]
where $e_i$ is the $i$th element of the canonical basis of $\Real^d$,
while $\mathcal{O}_d \in \Real^d$ and $\mathcal{O}_{d,d} \in \Real^{d \times d}$ respectively denote the zero (column) vector and matrix.
The boundary condition can be recast into the form $u\in\Ker(\mathcal{M}-\mathcal{N})$ setting
\[
\mathcal{M} \coloneqq \chi_{\Gamma_\dirichlet} \mathcal{M}_\dirichlet + \chi_{\Gamma_\neumann} \mathcal{M}_\neumann,
\qquad
\mathcal{M}_\dirichlet \coloneqq
\begin{bmatrix}
  \mathcal{O}_{d,d} & -n\\
  n^\trans & 0
\end{bmatrix},
\qquad
\mathcal{M}_\neumann \coloneqq
\begin{bmatrix}
  \mathcal{O}_{d,d} & n\\
  -n^\trans & 0
\end{bmatrix},
\]
where, for any subset $Z$ of $\Real^d$, $\chi_Z$ denotes its indicator function.

\subsubsection{Vector diffusion-advection-reaction}\label{subsubsec:vector-dar}

Assume $d = 3$ and let $\partial\Omega=\Gamma_\dirichlet\cup\Gamma_\neumann$ be partitioned as before.
Consider a positive diffusion coefficient $\varepsilon\in W^{1,\infty}(\Omega)$ and a scalar reaction coefficient $\gamma\in L^\infty(\Omega)$.
Let $\beta$ as before and define the inflow boundary $\Gamma^{-} \coloneqq  \left\{x\in\partial\Omega: \beta\cdot n(x) <0 \right\}$ as well as the outflow boundary $\Gamma^{+} \coloneqq \partial\Omega\setminus\Gamma^{-}$.
For a given source $f_p \in L^2(\Omega;\Real^3)$, we consider the problem of finding the vector potential $p:\Omega\to\Real^3$ such that:
\begin{equation}\label{eq:vector_DAR}
  \begin{alignedat}{2}
    \nabla\times(\varepsilon\nabla\times p) + L_\beta p + \gamma p
    &= f_p
    &\qquad& \text{in $\Omega$},
    \\
    n\times p + \chi_{\Gamma^-} (p\cdot n) n
    &= 0
    &\qquad& \text{on $\Gamma_\dirichlet$},
    \\
    n \times(\varepsilon \nabla \times p) + \chi_{\Gamma^-}(\beta\cdot n) p
    &= 0
    &\qquad& \text{on $\Gamma_\neumann$},
  \end{alignedat}
\end{equation}
where the Lie derivative operator $L_\beta$ is such that, for any function $q : \Omega \to \Real^3$ smooth enough,
\[
L_\beta q \coloneqq - \beta \times (\nabla \times q) + \nabla(\beta\cdot q).
\]
Introduce the flux $b \coloneqq \varepsilon\nabla \times p$ and the global variable $z \coloneqq (b,p)$.
Consider the identities
\[
\nabla \times s = \sum_{j=1}^{3} \mathcal{R}^j \partial_j s,
\qquad
\alpha\times s = \left(
\sum_{j=1}^{3} \mathcal{R}^j \alpha_j
\right) s \eqqcolon \mathcal{V}_\alpha s,
\]
where the family $(\mathcal{R}^j)_{1\le j\le 3}$ of anti-symmetric matrices $\mathcal{R}^j \in \Real^{3 \times 3}$ is such that $\mathcal{R}^{j}_{ik} = \epsilon_{ijk}$ for all $1 \le i,j,k\le 3$, $\epsilon$ being the Levi--Civita symbol.
Exploiting the vector identity
\[
\nabla(\beta\cdot p) = \beta\times(\nabla\times p) + p \times(\nabla\times \beta) + (p \cdot \nabla) \beta + (\beta \cdot \nabla) p,
\]
the problem can be rewritten in the form~\eqref{eq:strong} setting
$f \coloneqq
\begin{bmatrix}
  \mathcal{O}_3 \\ f_p
\end{bmatrix}$ and
\begin{equation}\label{eq:vector_DAR.as.fried}
  \mathcal{K} \coloneqq
  \begin{bmatrix}
    \varepsilon^{-1} \Id{3} & \mathcal{O}_{3,3} \\
    \mathcal{O}_{3,3} & (\nabla \beta - \mathcal{V}_{\nabla\times\beta} +\gamma\Id{3})
  \end{bmatrix},
  \quad
  \mathcal{A}^i \coloneqq
  \begin{bmatrix}
    \mathcal{O}_{3,3} & -\mathcal{R}^i\\
    -(\mathcal{R}^i)^\trans & \beta_i\Id{3}
  \end{bmatrix},
  \quad
  \mathcal{N} \coloneqq
  \begin{bmatrix}
    \mathcal{O}_{3,3} & -\mathcal{V}_n \\
    \mathcal{V}_n & (\beta\cdot n) \Id{3}
  \end{bmatrix}.
\end{equation}
The boundary condition can be recast in the form $z\in\Ker(\mathcal{N}-\mathcal{M})$ setting
\[
\mathcal{M} \coloneqq \chi_{\Gamma_\dirichlet} \mathcal{M}_\dirichlet + \chi_{\Gamma_\neumann} \mathcal{M}_\neumann,
\qquad
\mathcal{M}_\dirichlet \coloneqq
\begin{bmatrix}
  \mathcal{O}_{3,3} & \mathcal{V}_n\\
  \mathcal{V}_n & |\beta\cdot n|\Id{3}
\end{bmatrix},
\qquad
\mathcal{M}_\neumann \coloneqq
\begin{bmatrix}
  \mathcal{O}_{3,3} & -\mathcal{V}_n\\
  -\mathcal{V}_n & |\beta\cdot n|\Id{3}
\end{bmatrix}.
\]
To guarantee the validity of~\eqref{eq:lower.bound.K+KH-X},
we suppose the existence of $\Kmin > 0$ such that, almost everywhere in $\Omega$,
\[
\left(
\gamma - \frac{\nabla\cdot\beta}{2}
\right)\Id{3} + \nabla_\symm \beta \ge \Kmin \Id{3},
\]
where $\nabla_{\rm s}$ stands for the symmetric part of the gradient.

\begin{remark}[Interpretation of problem~\eqref{eq:vector_DAR}]
  In the context of magnetohydrodyamics, problem~\eqref{eq:vector_DAR} can be used to describe the magnetic field $b$ in terms of a vector potential $p$ (see~\cite[Section~5]{Heumann.Hiptmair:11} as well as~\cite{Heumann.Hiptmair:13}).
  In Section~\ref{subsec:magn.field.diff.conv}, the equation is adapted with minor changes to a context where the potential variable $p$ has the meaning of magnetic field and the flux $b$ is proportional to a current density.
\end{remark}


\section{Discrete setting}\label{sec:discrete.setting}

\subsection{Mesh}

We consider a polytopal mesh $(\Th,\Fh)$ matching the geometrical requirements detailed in~\cite[Definition 1.4]{Di-Pietro.Droniou:20}, with $\Th$ set of elements and $\Fh$ set of faces.
For any $Y \in \Th \cup \Fh$, we denote by $h_Y$ its diameter, so that $h=\max_{T\in\Th}h_T>0$.
Boundary faces lying on $\partial\Omega$  are collected in the set $\Fhb$ and we set $\Fhi \coloneqq \Fh \setminus \Fhb$.
For every mesh element $T\in\Th$, we denote by $\FT$ the subset of $\Fh$ containing the faces that lie on the boundary $\partial T$ of $T$.
For all $F \in \Fhb$, $n_F \coloneqq n$ denotes the unit normal vector to $F$ pointing out of $\Omega$.
For any $F \in \Fhi$, on the other hand, we fix a unique orientation by selecting once and for all a unit normal vector $n_F$.
For all $T \in \Th$ and all $F \in \FT$, we denote by $\omega_{TF}$ the orientation of $F$ relative to $T$ such that $\omega_{TF} n_F$ points out of $T$.

\subsection{Polynomial spaces}

Given $Y \in \Th \cup \Fh$ and an integer $\ell \ge 0$, we denote by $\Poly{\ell}(Y;\mathbb{X})$, $\mathbb{X} \in \{ \Complex^m, \Complex^{m\times m} \}$, the space spanned by the restriction to $Y$ of polynomials in the space variables of total degree $\le \ell$ with $\mathbb{X}$-valued coefficients.
The $L^2$-orthogonal projector on $\Poly{\ell}(Y;\mathbb{X})$ is denoted by $\lproj{\ell}{Y}$ irrespectively of the value of $\mathbb{X}$, all ambiguity being removed by the context.

\subsection{Trace operators on faces and local integration by parts formula}

For all $F \in \Fh$, denoting by $(n_{F,i})$ the Cartesian components of $n_F$, let
\[
\mathcal{N}_F \coloneqq \sum_{i = 1}^d n_{F,i} \mathcal{A}^i_{|F}.
\]
By the assumed regularity on $\mathcal{A}$, this quantity is single-valued on $F$ and it satisfies $\mathcal{N}_F \in L^\infty(F;\Complex^{m\times m})$.

For all $Y \in \Th \cup \Fh$, denote by $(w,v)_Y \coloneqq \int_Y (w,v)_{\Complex^m}$ the inner product of $L^2(Y;\Complex^m)$.
Let $T \in \Th$.
The following integration by parts formula corresponds to the first display equation in the proof of~\cite[Lemma~60.1]{Ern.Guermond:21*2}:
For all $(w,v) \in H^s(T;\Complex^m)^2$, $s > \frac12$,
\begin{equation}\label{eq:ibp:local}
  (A w, v)_T
  = (w, \tilde{A} v)_T + \sum_{F \in \FT} \omega_{TF} (\mathcal{N}_F w, v)_F
  = (w, \tilde{A} v)_T + \sum_{F \in \FT} \omega_{TF} (w, \mathcal{N}_F v)_F,
\end{equation}
where, to obtain the second equality, we have written $(\mathcal{N}_F w, v)_F = \int_F (\mathcal{N}_F w, v)_{\Complex^m} = \int_F v^\herm \mathcal{N}_F w = \int_F v^\herm \mathcal{N}_F^\herm w = \int_F (w, \mathcal{N}_F v)_{\Complex^m} = (w, \mathcal{N}_F v)_F$ (the third step being a consequence of the fact that $\mathcal{N}_F$ is a Hermitian field).


\section{A hybrid scheme for Friedrichs systems}\label{sec:scheme}

\subsection{Spaces and interpolators}

Let an integer $k \ge 0$ be fixed.
We consider discretizations of problem~\eqref{eq:weak} based on the following space:
\[
\underline{U}_h^k \coloneqq \left\{
\underline{v}_h = ( (v_T)_{T \in \Th}, (v_F)_{F \in \Fh} )
\st
\text{%
  $v_Y \in \Poly{k}(Y;\Complex^m)$ for all $Y \in \Th \cup \Fh$
}
\right\}.
\]
Let $s > \frac12$.
The meaning of the components is provided by the interpolator $\Ih : H^s(\Omega;\Complex^m) \to \Uh$ such that, for all $v \in H^s(\Omega;\Complex^m)$,
\[
\Ih v \coloneqq
\big(
(\lproj{k}{T} v)_{T \in \Th},
(\lproj{k}{F} v)_{F \in \Fh}
\big).
\]
For all $T \in \Th$, the restrictions of $\Uh$, $\underline{v}_h \in \Uh$, and $\Ih$ to $T$, obtained collecting the components attached to $T$ and its faces, are denoted replacing the subscript $h$ with $T$.
Finally, for all $\underline{v}_h \in \Uh$, we let $v_h \in L^2(\Omega;\Complex^m)$ be the piecewise polynomial field such that
\begin{equation}\label{eq:vh}
  (v_h)_{|T} \coloneqq v_T \qquad \forall T \in \Th.
\end{equation}

\subsection{Global discrete integration by parts formula}

Straightforward algebraic manipulations show that, for any Hermitian matrix $H \in \Complex^{m\times m}$ and vectors $\alpha_1,\,\beta_1,\,\alpha_2,\,\beta_2 \in \Complex^m$, it holds
\begin{equation}\label{eq:magic:algebraic}
  (H \alpha_1, \beta_1)_{\Complex^m} - (H \alpha_2, \beta_2)_{\Complex^m}
  = \left(\frac{\alpha_1 + \alpha_2}{2}, H(\beta_1 - \beta_2)\right)_{\Complex^m}
  + \left(H(\alpha_1 - \alpha_2), \frac{\beta_1 + \beta_2}{2}\right)_{\Complex^m}.
\end{equation}

\begin{proposition}[Global discrete integration by parts formula]\label{prop:ibp:global}
  For all $(\underline{w}_h, \underline{v}_h) \in \Uh \times \Uh$, it holds
  \begin{equation}\label{eq:ibp:global}
    \begin{aligned}
      \sum_{T \in \Th} (A w_T, v_T)_T
      &= \sum_{T \in \Th} (w_T, \tilde{A} v_T)_T
      \\
      &\quad
      - \sum_{T \in \Th} \sum_{F \in \FT} \omega_{TF} \left(\frac{w_F + w_T}{2}, \mathcal{N}_F (v_F - v_T)\right)_F
      \\
      &\quad
      - \sum_{T \in \Th} \sum_{F \in \FT} \omega_{TF}  \left(\mathcal{N}_F(w_F - w_T), \frac{v_F + v_T}{2}\right)_F
      \\
      &\quad
      + \sum_{F \in \Fhb} (\mathcal{N}_F w_F, v_F)_F.
    \end{aligned}
  \end{equation}
\end{proposition}

\begin{proof}
  By the single-valuedness of $(\mathcal{N}_F w_F, v_F)_F$ at interfaces together with the fact that $\omega_{T_1F} + \omega_{T_2F} = 0$ whenever $F \in \Fhi$ is shared by the mesh elements $T_1$ and $T_2$, we infer
  \[
  0 = \sum_{F \in \Fhb} (\mathcal{N}_F w_F, v_F)_F
  - \sum_{T \in \Th} \sum_{F \in \FT} \omega_{TF} (\mathcal{N}_F w_F, v_F)_F.
  \]
  Summing the local integration by parts formula~\eqref{eq:ibp:local} with $(w,v) = (w_T, v_T)$ over $T \in \Th$ and adding the above expression, we obtain
  \begin{equation}\label{eq:ibp:global:intermediate}
    \begin{aligned}
      \sum_{T \in \Th} (A w_T, v_T)_T
      &= \sum_{T \in \Th} (w_T, \tilde{A} v_T)_T
      \\
      &\quad
      - \sum_{T \in \Th} \sum_{F \in \FT} \omega_{TF} \left[
        (\mathcal{N}_F w_F, v_F)_F
        - (\mathcal{N}_F w_T, v_T)_F
        \right]
      \\
      &\quad
      + \sum_{F \in \Fhb} (\mathcal{N}_F w_F, v_F)_F.
    \end{aligned}
  \end{equation}
  For all $T \in \Th$ and all $F \in \FT$, using~\eqref{eq:magic:algebraic}, we infer, for all $a_1,\,b_1,\,a_2,\,b_2 \in L^2(F;\Complex^m)$,
  \begin{equation}\label{eq:magic.formula}
    (\mathcal{N}_F a_1, b_1)_F
    - (\mathcal{N}_F a_2, b_2)_F
    = \left(\frac{a_1 + a_2}{2}, \mathcal{N}_F (b_1 - b_2)\right)_F
    + \left(\mathcal{N}_F (a_1 - a_2), \frac{b_1 + b_2}{2}\right)_F.
  \end{equation}
  The conclusion follows applying the above relation with $(a_1,a_2,b_1,b_2) = (w_F, w_T, v_F, v_T)$ to reformulate the second term in the right-hand side of~\eqref{eq:ibp:global:intermediate}.
\end{proof}

\subsection{Reference quantities and notation for inequalities up to a constant}\label{sec:scheme:inequalities}

Recalling~\eqref{eq:lower.bound.K+KH-X}, we define, for all $T \in \Th$, the following quantities:
\begin{gather}\label{eq:Aref.tref}
  \Aref
  \coloneqq
  \max_{1\le i \le d} \norm{L^\infty(T; \Complex^{m\times m})}{\mathcal{A}^i},
  \\ \label{eq:tauT}
  \tref
  \coloneqq
  \max(\Aref h_T^{-1}, \Kmin)^{-1}
  = \min(\Aref^{-1} h_T, \Kmin^{-1}),
\end{gather}
and notice that, passing to the $L^\infty$-norm in~\eqref{eq:NF} and recalling that $n_F$ is a unit vector,
$\norm{L^\infty(F;\Complex^{m\times m})}{\mathcal{N}_F} \le d \Aref$ for all $F \in \FT$.
Moreover, recalling~\cite[Eq.~(60.27)]{Ern.Guermond:21*2}, there is $C_{\mathcal{A}} > 0$ such that, for all $T \in \Th$,
\begin{equation}\label{eq:bnd.pi.0.T.Ai-Ai}
  \norm{L^\infty(T;\Complex^{m\times m})}{\lproj{0}{T} \mathcal{A}^i - \mathcal{A}^i}
  \le C_{\mathcal{A}} \left(
  \Kmin \Aref h_T
  \right)^{\frac12}.
\end{equation}

From this point on, in order to alleviate the notation, we will write $a \lesssim b$ for the inequality $a\le Cb$ with $C$ possibly depending only on $d$, $\Omega$, the mesh regularity parameter of~\cite[Definition 1.9]{Di-Pietro.Droniou:20}, $C_{\mathcal{A}}$ in~\eqref{eq:bnd.pi.0.T.Ai-Ai}, and polynomial degrees involved in the expressions.
We stress that this means, in particular, that $C$ is independent of the mesh size, and, for local inequalities, on the mesh element or face on which they hold.
Additionally, no other dependence on the problem data $\mathcal{A}$, $\mathcal{K}$, and $\mathcal{M}$ is allowed.

\subsection{Discrete sesquilinear form}

We define the discrete sesquilinear form $a_h : \Uh \times \Uh \to \Complex$ such that, for all $(\underline{w}_h, \underline{v}_h) \in \Uh \times \Uh$,
\begin{equation}\label{eq:ah}
  \begin{aligned}
    a_h(\underline{w}_h, \underline{v}_h)
    &\coloneqq
    \sum_{T \in \Th} (A w_T, v_T)_T
    + \Kmin \sum_{T \in \Th} h_T \sum_{F \in \FT} (w_F - w_T, v_F - v_T)_F
    \\
    &\quad
    + \sum_{T \in \Th} \sum_{F \in \FT} \omega_{TF} \left(\mathcal{N}_F (w_F - w_T), \frac{v_F + v_T}{2}\right)_F
    \\
    &\quad
    + \frac12 \sum_{F \in \Fhb} ((\mathcal{M}_F + \mathcal{S}_F^{\rm b} - \mathcal{N}_F) w_F, v_F)_F
    + j_h(\underline{w}_h, \underline{v}_h),
  \end{aligned}
\end{equation}
with stabilization sesquilinear form $j_h : \Uh \times \Uh \to \Complex$ such that
\begin{equation}\label{eq:sh}
  j_h(\underline{w}_h, \underline{v}_h)
  \coloneqq
  \sum_{T \in \Th} \sum_{F \in \FT} (\mathcal{S}_{TF}^{\rm i} (w_F - w_T), v_F - v_T)_F.
\end{equation}
Above, $\{ \mathcal{S}_F^{\rm b} \}_{F \in \Fhb}$ and $\{ \mathcal{S}_{TF}^{\rm i} \}_{T \in \Th,\, F \in \FT}$ are two families of positive Hermitian operators the goal of which is to control boundary and interface jumps.

\begin{assumption}[Stabilization]\label{ass:stabilization}
  For all $T \in \Th$ and for all $F \in \FT$, the Hermitian positive field $\mathcal{S}_{TF}^{\rm i} \in L^{\infty}(F;\Complex^{m\times m})$ satisfies
  \begin{alignat}{4}
    \label{eq:STFi:boundedness}
    (\mathcal{S}_{TF}^{\rm i} v, v)_F
    &\lesssim \Aref \norm{L^2(F;\Complex^m)}{v}^2
    &\qquad& \forall v \in \Poly{k}(F;\Complex^m),
    \\ \label{eq:STFi:control.NF}
    |(w,\mathcal{N}_F v)_F|
    &\lesssim \Aref^{\frac12} \norm{L^2(F;\Complex^m)}{w}~(\mathcal{S}_{TF}^{\rm i} v, v)_F^{\frac12}
    &\qquad& \forall (w,v) \in L^2(F;\Complex^m)^2,
  \end{alignat}
  and, for all $F \in \Fhb$, denoting by $T_F$ the unique element of $\Th$ such that $F \in \mathcal{F}_{T_F}$, the Hermitian positive field $\mathcal{S}_F^{\rm b} \in L^{\infty}(F;\Complex^{m\times m})$ is such that
  \begin{alignat}{4}\label{eq:SFb:stabilization}
    \Ker (\mathcal{M}_F - \mathcal{N}_F)
    &\subset \Ker (\mathcal{M}_F +  \mathcal{S}_F^{\rm b} - \mathcal{N}_F),
    \\ \label{eq:SFb:control.NF}
    |((\mathcal{M}_F + \mathcal{S}_F^{\rm b} + \mathcal{N}_F) w, v)_F|
    &\lesssim \Aref[T_F]^{\frac12}~\norm{L^2(F;\Complex^m)}{w}~((\mathcal{M}_F + \mathcal{S}_F^{\rm b}) v, v)_F^{\frac12}
    &\qquad& \forall (w,v) \in L^2(F;\Complex^m)^2.
  \end{alignat}
\end{assumption}


\begin{remark}[Stabilization for scalar and vector diffusion-advection-reaction]\label{rmk:penalty-fields}
  Following~\cite[Example~60.14]{Ern.Guermond:21*2}, penalty fields fulfilling Assumption~\ref{ass:stabilization} for the scalar diffusion-advection-reaction equation of Section~\ref{subsubsec:scalar-dar} can be designed setting, with $\Aref[T] \coloneqq \max(1, \|\beta\|_{L^\infty(T, \Real^d)})$,
  \[
  \mathcal{S}_{TF}^{\rm i} =
  \begin{bmatrix}
    \Aref[T] n_F\otimes n_F & \mathcal{O}_{d}
    \\
    \mathcal{O}_{d}^\trans & |\beta\cdot n|
  \end{bmatrix}
  ,
  \qquad
  \mathcal{S}_{F}^{\rm b} = \Aref[T]
  \begin{bmatrix}
    \mathcal{O}_{d d} & \mathcal{O}_{d}
    \\
    \mathcal{O}_{d}^\trans & 1
  \end{bmatrix}.
  \]
  For the vector diffusion-advection-reaction equation of Section~\ref{subsubsec:vector-dar}, following~\cite[Example 60.15]{Ern.Guermond:21*2}, penalty fields can be taken, with $\Aref[T] \coloneqq \max(1, \norm{L^\infty(T, \Real^3)}{\beta})$, equal to
  \[
  \mathcal{S}_{TF}^{\rm i} = \Aref[T]
  \begin{bmatrix}
    \mathcal{V}_n^\trans \mathcal{V}_n & \mathcal{O}_{dd}
    \\
    \mathcal{O}_{dd}^\trans & \mathcal{V}_n^\trans\mathcal{V}_n
  \end{bmatrix}
  ,
  \qquad
  \mathcal{S}_{F}^{\rm b} = \Aref[T]
  \begin{bmatrix}
    \mathcal{O}_{d d} & \mathcal{O}_{dd}
    \\
    \mathcal{O}_{dd}^\trans & \mathcal{V}_n^\trans\mathcal{V}_n
  \end{bmatrix}.
  \]
\end{remark}

\subsection{Discrete problem}

The discrete problem reads: Find $\underline{u}_h \in \Uh$ such that
\begin{equation}\label{eq:discrete}
  a_h(\underline{u}_h, \underline{v}_h) = (f, v_h)_\Omega
  \qquad \forall \underline{v}_h \in \Uh,
\end{equation}
where $v_h$ is defined by~\eqref{eq:vh}.
Some remarks are in order.

\begin{remark}[Static condensation]
  The algebraic problem corresponding to~\eqref{eq:discrete} can be efficiently solved by first statically condensing the element unknowns element by element and then solving a global system in the face unknowns only.
  Such a strategy leads, in many cases, to algebraic problems that are significantly smaller than the ones for DG methods; see, e.g.,~\cite{Botti.Di-Pietro:22} on this subject, where a detailed comparison of Hybrid High-Order and DG methods for the Stokes problem in three space dimensions is provided.
\end{remark}

\begin{remark}[Flux formulation]
  For an element $T\in\Th$ and a face $F\in\FT$, define the numerical flux $\Phi_{TF} : \UT \to L^2(F; \Complex^m)$ such that, for all $\underline{v}_T \in \UT$,
  \[
  \Phi_{TF} (\underline{v}_T) \coloneqq
  \begin{cases}
    \lproj{k}{F} \left[\omega_{TF} \mathcal{N}_F\left( \frac{v_F +v_T}{2}\right) - (\Kmin \Id{m} + \mathcal{S}_{TF}^{\rm i})(v_F-v_T)\right]
    &\quad F\in \Fhi,
    \\
    \lproj{k}{F} \left[\mathcal{N}_F\left( \frac{v_F +v_T}{2}\right) - (\Kmin \Id{m} + \mathcal{S}_{TF}^{\rm i})(v_F-v_T) - \frac12 (\mathcal{M}_F + \mathcal{N}_F + \mathcal{S}^{\rm b}_{TF}) v_F \right]
    &\quad F\in \Fhb.
  \end{cases}
  \]
  For $F \in \Fhi$, denote by $T_1$ and $T_2$ the elements sharing $F$.
  Problem~\eqref{eq:discrete} admits the following equivalent flux formulation:
  Find $\underline{u}_h\in\Uh$ such that
  \begin{subequations}\label{eq:discrete:flux}
    \begin{alignat}{2}\label{eq:discrete:flux:local.balance}
      (u_T, \tilde{A} v_T)_T + \sum_{F\in\FT} (\Phi_{TF}(\underline{u}_T), v_T)_F &= (f, v_T)_T
      &\qquad& \forall T\in\Th, \, \forall v_T \in \Poly{k}(T),
      \\ \label{eq:discrete:flux:continuity}
      \Phi_{T_1F}(\underline{u}_{T_1}) + \Phi_{T_2F} (\underline{u}_{T_2}) &= 0
      &\qquad& \forall F\in \Fhi,
      \\ \label{eq:discrete:flux:boundary}
      \Phi_{T_F F}(\underline{u}_{T_F}) &= 0
      &\qquad& \forall F\in \Fhb.
    \end{alignat}
  \end{subequations}
\end{remark}

\begin{remark}[Comparison with the scheme of~\cite{Chen.Kang.ea:24}]\label{rem:Chen.Kang.ea:24:comparison}
  In~\cite{Chen.Kang.ea:24}, an $hp$-hybridizable DG scheme for Friedrichs systems with
  a similar choice of spaces as the one considered here is proposed.
  The authors present an upwind-based numerical flux (see~\cite[Eq. 3.5]{Chen.Kang.ea:24}) which, for $F\in \Fhi$, can be rephrased with the notation of this work as
  \[
  \Phi_{TF} (\underline{v}_T) \coloneqq  \lproj{k}{F} \left[ \omega_{TF} \mathcal{N}_F v_T - |\mathcal{N}_F|(v_F-v_T) \right],
  \]
  where $|\mathcal{N}_F|$ is well-defined since $\mathcal{N}_F$ is Hermitian.
  Our analysis can be adapted with very minor adjustments to cover this option of numerical flux,
  thereby completing the partial coercivity results of~\cite{Chen.Kang.ea:24} with inf-sup stability and convergence results.
  In particular, it is readily verified that the choice $\mathcal{S}_{TF}^{\rm i} = |\mathcal{N}_F|$ fulfils Assumption~\ref{ass:stabilization}.
\end{remark}


\section{Stability and error analysis}\label{sec:analysis}

In this section we carry out a stability and convergence analysis for the numerical scheme~\eqref{eq:discrete}.

\subsection{Stability}

\begin{proposition}[Equivalent reformulation of $a_h$]\label{eq:ah:reformulation.stability}
  For all $(\underline{w}_h, \underline{v}_h) \in \Uh \times \Uh$, it holds
  \begin{equation}\label{eq:ah:reformulation.stability}
    \begin{aligned}
      a_h(\underline{w}_h, \underline{v}_h)
      &=
      \frac12 \sum_{T \in \Th} \left[
        (A w_T, v_T)_T
        + (w_T, \tilde{A} v_T)_T
        \right]
      + \Kmin \sum_{T \in \Th} h_T \sum_{F \in \FT} (w_F - w_T, v_F - v_T)_F
      \\
      &\quad
      + \frac12 \sum_{T \in \Th} \sum_{F \in \FT} \omega_{TF} \left[
        (\mathcal{N}_F w_F, v_T)_F
        - (\mathcal{N}_F w_T, v_F)_F
        \right]
      \\
      &\quad
      + \frac12 \sum_{F \in \Fhb} ((\mathcal{M}_F + \mathcal{S}_F^{\rm b}) w_F, v_F)_F
      + j_h(\underline{w}_h, \underline{v}_h).
    \end{aligned}
  \end{equation}
\end{proposition}

\begin{proof}
  Write the first term in the right-hand side of~\eqref{eq:ah} as $\sum_{T \in \Th} (A w_T, v_T)_T = \frac12 \sum_{T \in \Th} (A w_T, v_T)_T + \frac12 \sum_{T \in \Th} (A w_T, v_T)_T$ and apply~\eqref{eq:ibp:global} to the second addend to obtain
  \[
  \begin{aligned}
    a_h(\underline{w}_h, \underline{v}_h)
    &=
    \frac12 \sum_{T \in \Th} \left[
      (A w_T, v_T)_T
      + (w_T, \tilde{A} v_T)_T
      \right]
    + \Kmin \sum_{T \in \Th} h_T \sum_{F \in \FT} (w_F - w_T, v_F - v_T)_F
    \\
    &\quad
    + \frac14 \sum_{T \in \Th} \sum_{F \in \FT} \omega_{TF} \left[
      (\mathcal{N}_F(w_F - w_T), v_F + v_T)_F
      - (w_F + w_T, \mathcal{N}_F (v_F - v_T))_F
      \right]
    \\
    &\quad
    + \frac12 \sum_{F \in \Fhb} ((\mathcal{M}_F + \mathcal{S}_F^{\rm b}) w_F, v_F)_F
    + j_h(\underline{w}_h, \underline{v}_h).
  \end{aligned}
  \]
  Applying~\eqref{eq:magic.formula} with $(a_1,b_1,a_2,b_2) = \left(w_F - w_T, v_F + v_T, w_F + w_T, v_F - v_T\right)$ to the terms in the second line after using the fact that $\mathcal{N}_F$ is Hermitian, the conclusion follows.
\end{proof}

We define the seminorms and norms such that, for all $\underline{v}_h \in \Uh$,
\begin{equation}\label{eq:seminorm.S.b}
  \seminorm{{\rm b},h}{\underline{v}_h}^2
  \coloneqq
  \sum_{F \in \Fhb} ((\mathcal{M}_F + \mathcal{S}_F^{\rm b}) v_F, v_F)_F,
  \qquad
  \seminorm{j,h}{\underline{v}_h}
  \coloneqq j_h(\underline{v}_h, \underline{v}_h),
\end{equation}
\begin{equation}\label{eq:tnorm.h}
  \begin{aligned}
    \tnorm{\flat,h}{\underline{v}_h}^2
    &\coloneqq \Kmin \sum_{T \in \Th} \norm{0,T}{\underline{v}_T}^2
    + \frac12 \seminorm{{\rm b},h}{\underline{v}_h}^2
    + \seminorm{j,h}{\underline{v}_h}^2,
    \\
    \tnorm{h}{\underline{v}_h}^2
    &\coloneqq
    \tnorm{\flat,h}{\underline{v}_h}^2
    + \sum_{T \in \Th} \tref \norm{L^2(T;\Complex^m)}{A_1 v_T}^2,
  \end{aligned}
\end{equation}
where
\begin{equation}\label{eq:norm.0.T}
  \norm{0,T}{\underline{v}_T}^2
  \coloneqq
  \norm{L^2(T;\Complex^m)}{v_T}^2
  + \seminorm{0,\partial T}{\underline{v}_T}^2,
  \qquad
  \seminorm{0,\partial T}{\underline{v}_T}^2
  \coloneqq h_T \sum_{F \in \FT} \norm{L^2(F;\Complex^m)}{v_F - v_T}^2.
\end{equation}

\begin{corollary}[Partial coercivity of $a_h$]\label{cor:ah:partial.coercivity}
  Denoting by $\Re:\Complex \to \Real$ the real part operator, for all $\underline{v}_h \in \Uh$, it holds
  \begin{equation}\label{eq:ah:partial.coercivity}
    \tnorm{\flat,h}{\underline{v}_h}^2
    \le
    \frac12 ((\mathcal{K} + \mathcal{K}^\herm - \nabla \cdot \mathcal{A}) v_h, v_h)_\Omega
    + \Kmin \sum_{T \in \Th} \seminorm{0,\partial T}{\underline{v}_T}^2
    + \frac12 \seminorm{{\rm b},h}{\underline{v}_h}^2
    + \seminorm{j,h}{\underline{v}_h}^2
    = \Re(a_h(\underline{v}_h, \underline{v}_h)).
  \end{equation}
\end{corollary}

\begin{proof}
  The inequality in~\eqref{eq:ah:partial.coercivity} is a straightforward consequence of the lower bound~\eqref{eq:lower.bound.K+KH-X} for $\mathcal{K} + \mathcal{K}^\herm - \nabla \cdot \mathcal{A}$ together with the definition~\eqref{eq:tnorm.h} of the $\tnorm{\flat,h}{\cdot}$-norm.
  To prove the equality, we start by writing~\eqref{eq:ah:reformulation.stability} for $\underline{w}_h = \underline{v}_h$ and notice that, for all $T \in \Th$,
  \[
  \begin{aligned}
    (A v_T, v_T)_T
    + (v_T, \tilde{A} v_T)_T
    &= \int_T v_T^\herm A v_T
    + \int_T (\tilde{A} v_T)^\herm v_T
    \\
    \overset{\eqref{eq:A1.A.tA}}&=
    \int_T v_T^\herm A v_T
    + \int_T v_T^\herm (\mathcal{K} + \mathcal{K}^\herm - \nabla \cdot \mathcal{A} - A)^\herm v_T
    \\
    &=
    ((A - A^\herm) v_T, v_T)_T
    + ((\mathcal{K} + \mathcal{K}^\herm - \nabla \cdot \mathcal{A}) v_T, v_T)_T,
  \end{aligned}
  \]
  where we have used the fact that $\mathcal{K} + \mathcal{K}^\herm - \nabla \cdot \mathcal{A}$ is a Hermitian field to conclude.
  The desired result then follows noticing that $\Re(((A - A^\herm) v_T, v_T)_T) = 0$ since $(A - A^\herm)$ is skew-Hermitian by definition.
\end{proof}

\begin{lemma}[Inf-sup condition]\label{lem:inf-sup}
  Let Assumptions~\ref{ass:stabilization} hold.
  Then, for all $\underline{v}_h \in \Uh$, it holds
  \begin{equation}\label{eq:inf-sup}
    \tnorm{h}{\underline{v}_h}
    \lesssim \rho^2 \sup_{\underline{w}_h \in \Uh \setminus \{ \underline{0} \}}\frac{
      |a_h(\underline{v}_h, \underline{w}_h)|
    }{\tnorm{h}{\underline{w}_h}},
  \end{equation}
  where
  \begin{equation}\label{eq:rho}
    \text{%
      $\rho \coloneqq \frac{\Kmax}{\Kmin}$
      with $\Kmax \coloneqq \norm{L^\infty(\Omega;\Complex^{m\times m})}{\mathcal{K}} + \norm{L^\infty(\Omega;\Complex^{m\times m})}{\nabla \cdot \mathcal{A}}$.
    }
  \end{equation}
  As a consequence, problem~\eqref{eq:discrete} is well-posed.
\end{lemma}

\begin{proof}
  The well-posedness of problem~\eqref{eq:discrete} is a classical consequence of the inf-sup condition~\eqref{eq:inf-sup}, so we only detail the proof of the latter.
  To this purpose, let $\underline{v}_h \in \Uh$.
  \medskip \\
  \underline{(i) \emph{Test function for the control of the directional derivative}.}
  Define $\underline{z}_h \in \Uh$ such that
  \begin{equation}\label{eq:zh}
    \begin{alignedat}{2}
      \text{
        $z_T = \tref A_{1,T}^0 v_T$ for all $T \in \Th$
        and $z_F = 0$ for all $F \in \Fh$,
      }
    \end{alignedat}
  \end{equation}
  where $A_{1,T}^0 \coloneqq \sum_{i = 1}^d \lproj{0}{T} \mathcal{A}^i \partial_i$.
  We first prove that
  \begin{equation}\label{eq:zh:L.norm.estimate}
    \sum_{T \in \Th} \tref^{-1} \norm{L^2(T;\Complex^m)}{z_T}^2
    \lesssim
    \tnorm{h}{\underline{v}_h}^2.
  \end{equation}
  To this end, given $T \in \Th$, we use a triangle inequality to write
  \[
  \tref^{-\frac12} \norm{L^2(T;\Complex^m)}{z_T}
  \overset{\eqref{eq:zh}}\le
  \tref^{\frac12} \norm{L^2(T;\Complex^m)}{(A_{1,T}^0 - A_1) v_T}
  + \tref^{\frac12} \norm{L^2(T;\Complex^m)}{A_1 v_T}.
  \]
  For the first term in the right-hand side, using first $(\infty,2)$-H\"older inequalities and then discrete inverse inequalities to write $\norm{L^2(T;\Complex^m)}{\partial_i v_T} \lesssim h_T^{-1} \norm{L^2(T;\Complex^m)}{v_T}$ for all $1\le i \le d$, we get
  \begin{equation}\label{eq:zh:tnorm.estimate:T1}
    \begin{aligned}
      \tref^{\frac12} \norm{L^2(T;\Complex^m)}{(A_{1,T}^0 - A_1) v_T}
      &\lesssim \tref^{\frac12}
      \, \norm{L^\infty(T;\Complex^{m\times m})}{\lproj{0}{T} \mathcal{A}^i - \mathcal{A}^i}
      \, h_T^{-1} \norm{L^2(T;\Complex^m)}{v_T}
      \\
      \overset{\eqref{eq:bnd.pi.0.T.Ai-Ai}}&\lesssim
      \tref^{\frac12} (\Aref h_T^{-1})^{\frac12}
      \, \Kmin^{\frac12} \norm{L^2(T;\Complex^m)}{v_T}
      \\
      \overset{\eqref{eq:tauT}}&\le
      \Kmin^{\frac12} \norm{L^2(T;\Complex^m)}{v_T},
    \end{aligned}
  \end{equation}
  so that $\tref^{-\frac12} \norm{L^2(T;\Complex^m)}{z_T} \lesssim \Kmin^{\frac12} \norm{L^2(T;\Complex^m)}{v_T} + \tref^{\frac12} \norm{L^2(T;\Complex^m)}{A_1 v_T}$.
  Squaring, using the fact that ${(a + b)^2} \le {2(a^2 + b^2)}$ for all $a,\,b \in \Real$ in the right-hand side, and recalling the definition~\eqref{eq:tnorm.h} of the triple norm proves~\eqref{eq:zh:L.norm.estimate}.

  We next show that
  \begin{equation}\label{eq:zh:tnorm.estimate}
    \tnorm{h}{\underline{z}_h}
    \lesssim \tnorm{h}{\underline{v}_h}
  \end{equation}
  by estimating each term in the definition~\eqref{eq:tnorm.h} of $\tnorm{h}{\underline{z}_h}$.
  For the first term, we write
  \[
  \begin{aligned}
    \Kmin \sum_{T \in \Th} \norm{0,T}{\underline{z}_T}^2
    \overset{\eqref{eq:norm.0.T}}=
    \Kmin \sum_{T \in \Th} \left(
    \norm{L^2(T;\Complex^m)}{z_T}^2 + h_T \norm{L^2(\partial T;\Complex^m)}{z_T}^2
    \right)
    \overset{\eqref{eq:tauT}}\lesssim
    \sum_{T \in \Th} \tref^{-1} \norm{L^2(T;\Complex^m)}{z_T}^2
    \overset{\eqref{eq:zh:L.norm.estimate}}
    \lesssim \tnorm{h}{\underline{v}_h}^2,
  \end{aligned}
  \]
  where, in the second step, we have additionally used a discrete trace inequality to write $h_T \norm{L^2(\partial T;\Complex^m)}{z_T}^2 \lesssim \norm{L^2(T;\Complex^m)}{z_T}^2$.
  We next notice that $\seminorm{{\rm b},h}{\underline{z}_h} = 0$ since $z_F = 0$ for all $F \in \Fhb$.
  Using the fact that $z_F = 0$ for all $F \in \Fh$ in~\eqref{eq:sh}, we get
  \[
  \begin{aligned}
    \seminorm{j,h}{\underline{z}_h}^2
    \overset{\eqref{eq:seminorm.S.b},\,\eqref{eq:sh}}&=
    \sum_{T \in \Th} \sum_{F \in \FT} (\mathcal{S}_{TF}^{\rm i} z_T, z_T)_F
    \\
    \overset{\eqref{eq:STFi:boundedness}}&\lesssim
    \sum_{T \in \Th} \Aref h_T^{-1} \norm{L^2(T;\Complex^m)}{z_T}^2
    \overset{\eqref{eq:tauT}}\le
    \sum_{T \in \Th} \tref^{-1} \norm{L^2(T;\Complex^m)}{z_T}^2
    \overset{\eqref{eq:zh:L.norm.estimate}}\lesssim
    \tnorm{h}{\underline{v}_h}^2,
  \end{aligned}
  \]
  where we have additionally used a discrete trace inequality in the first bound.
  Finally, for all $T \in \Th$, it holds, by $(\infty,2)$-H\"{o}lder inequalities,
  \[
  \begin{aligned}
    \sum_{T \in \Th} \tref \norm{L^2(T;\Complex^m)}{A_1 z_T}^2
    \overset{\eqref{eq:A1.A.tA}}&\lesssim
    \sum_{T \in \Th} \tref  \sum_{i = 1}^d \norm{L^\infty(T;\Complex^{m\times m})}{\mathcal{A}^i}^2 \norm{L^2(T;\Complex^m)}{\partial_i z_T}^2
    \\
    \overset{\eqref{eq:Aref.tref}}&\lesssim
    \sum_{T \in \Th} \tref (\Aref h_T^{-1})^2 \norm{L^2(T;\Complex^m)}{z_T}^2
    \overset{\eqref{eq:tauT}}\lesssim
    \sum_{T \in \Th} \tref^{-1} \norm{L^2(T;\Complex^m)}{z_T}^2
    \overset{\eqref{eq:zh:L.norm.estimate}}\lesssim
    \tnorm{h}{\underline{v}_h}^2,
  \end{aligned}
  \]
  where we have additionally used a discrete inverse inequality to write $\norm{L^2(T;\Complex^m)}{\partial_i z_T}^2 \lesssim h_T^{-2} \norm{L^2(T;\Complex^m)}{z_T}^2$ in the second step.
  This concludes the proof of~\eqref{eq:zh:tnorm.estimate}.
  \medskip \\
  \underline{(ii) \emph{Estimate of the directional derivative.}}
  Throughout the rest of the proof, we denote by $K > 0$ a generic constant the value of which can change at each occurrence and that additionally has the same dependencies as the hidden multiplicative constant in $\lesssim$; see Section~\ref{sec:scheme:inequalities}.
  Denote by $\$$ the supremum in the right-hand side of~\eqref{eq:inf-sup}.
  We start by noticing that
  \begin{equation}\label{eq:inf-sup:partial.coercivity}
    \tnorm{\flat,h}{\underline{v}_h}^2
    \overset{\eqref{eq:tnorm.h}}=
    \Kmin \sum_{T \in \Th} \norm{0,T}{\underline{v}_T}^2
    + \frac12 \seminorm{{\rm b},h}{\underline{v}_h}^2
    + \seminorm{j,h}{\underline{v}_h}^2
    \overset{\eqref{eq:ah:partial.coercivity}}\le \Re(a_h(\underline{v}_h, \underline{v}_h))
    \lesssim \$\, \tnorm{h}{\underline{v}_h}.
  \end{equation}
  It only remains to estimate $\term \coloneqq \sum_{T \in \Th} \tref \norm{L^2(T;\Complex^m)}{A_1 v_T}^2$.
  To this purpose, recalling the definitions~\eqref{eq:zh} of $\underline{z}_h$,~\eqref{eq:ah} of $a_h$, and~\eqref{eq:A1.A.tA} of $A$, we write
  \[
  \begin{aligned}
    \term
    &=
    a_h(\underline{v}_h, \underline{z}_h)
    - \sum_{T \in \Th} \left[
      (\mathcal{K} v_T, z_T)_T
      - \Kmin h_T \sum_{F \in \FT} (v_F - v_T, z_T)_F
      \right]
    \\
    &\quad
    + \sum_{T \in \Th} \tref (A_1 v_T, (A_1 - A_{1,T}^0) v_T)_T
    \\
    &\quad
    - \frac12 \sum_{T \in \Th} \sum_{F \in \FT} \omega_{TF} (\mathcal{N}_F (v_F - v_T), z_T)_F
    + \sum_{T \in \Th} \sum_{F \in \FT} (\mathcal{S}_{TF}^{\rm i}(v_F - v_T), z_T)_F.
  \end{aligned}
  \]
  We proceed to bound the terms $\term_1,\ldots,\term_5$ in the right-hand side.
  By definition of the supremum, we readily have
  \[
  \term_1 \le \$\, \tnorm{h}{\underline{z}_h}
  \overset{\eqref{eq:zh:tnorm.estimate}}\lesssim
  \$\,  \tnorm{h}{\underline{v}_h}.
  \]
  The second term is treated using $(\infty,2,2)$-H\"{o}lder and Cauchy--Schwarz inequalities and recalling the definition~\eqref{eq:rho} of $\rho$:
  \[
  \begin{aligned}
    \term_2
    &\le
    \sum_{T \in \Th} \Kmin^{\frac12} \left(
    \rho \norm{L^2(T;\Complex^m)}{v_T}
    + h_T^{\frac12} \sum_{F \in \FT} \norm{L^2(F;\Complex^m)}{v_F - v_T}
    \right) \,
    \Kmin^{\frac12} \left(
    \norm{L^2(T;\Complex^m)}{z_T}
    + h_T^{\frac12} \norm{L^2(\partial T;\Complex^m)}{z_T}
    \right)
    \\
    \overset{\eqref{eq:norm.0.T}}&\lesssim
    \rho
    \sum_{T \in \Th} \Kmin^{\frac12} \norm{0,T}{\underline{v}_T}
    ~ \Kmin^{\frac12} \norm{L^2(T;\Complex^m)}{z_T}
    \\
    \overset{\eqref{eq:inf-sup:partial.coercivity},\,\eqref{eq:zh:tnorm.estimate}}&\le
    \rho \left(
    \$\, \tnorm{h}{\underline{v}_h}
    \right)^{\frac12}~\tnorm{h}{\underline{v}_h}
    \le
    K \rho^2 \$\, \tnorm{h}{\underline{v}_h}
    + \frac{1}{12} \tnorm{h}{\underline{v}_h}^2,
  \end{aligned}
  \]
  where we have additionally used
  a discrete trace inequality together with the fact that $\rho \ge 1$ in the second step,
  a Cauchy--Schwarz inequality on the sum in the third step,
  while the conclusion follows from the generalized Young's inequality.
  For the third term, we use Cauchy--Schwarz inequalities to write
  \[
  \begin{aligned}
    \term_3
    &\lesssim \sum_{T \in \Th}
    \tref^{\frac12} \norm{L^2(T;\Complex^m)}{A_1 v_T}~
    \tref^{\frac12} \norm{L^2(T;\Complex^m)}{(A_1 - A_{1,T}^0) v_T}
    \\
    \overset{\eqref{eq:zh:tnorm.estimate:T1}}&\lesssim
    \sum_{T \in \Th} \tref^{\frac12} \norm{L^2(T;\Complex^m)}{A_1 v_T}~
    \Kmin^{\frac12} \norm{L^2(T;\Complex^m)}{v_T}
    \\
    &\le
    \left(
    \sum_{T \in \Th} \tref \norm{L^2(T;\Complex^m)}{A_1 v_T}^2
    \right)^{\frac12}
    \left(
    \sum_{T \in \Th} \Kmin \norm{L^2(T;\Complex^m)}{v_T}^2
    \right)^{\frac12}
    \le
    \frac12 \term
    + K \$\, \tnorm{h}{\underline{v}_h},
  \end{aligned}
  \]
  where, in the third step, we have used a Cauchy--Schwarz inequality on the sum and we have concluded using Young's inequality and~\eqref{eq:inf-sup:partial.coercivity}.
  For the fourth term, we use~\eqref{eq:STFi:control.NF} (after recalling that $\mathcal{N}_F$ is a Hermitian field to move it in front of the second argument) followed by a Cauchy--Schwarz inequality on the sums and a discrete trace inequality to write
  \[
  \begin{aligned}
    \term_4
    &\lesssim \seminorm{j,h}{\underline{v}_h}~
    \left(
    \sum_{T \in \Th} \Aref h_T^{-1} \norm{L^2(T;\Complex^m)}{z_T}^2
    \right)^{\frac12}
    \\
    \overset{\eqref{eq:tauT}}&\lesssim
    \seminorm{j,h}{\underline{v}_h}~
    \left(
    \sum_{T \in \Th} \tref^{-1} \norm{L^2(T;\Complex^m)}{z_T}^2
    \right)^{\frac12}
    \\
    \overset{\eqref{eq:zh:L.norm.estimate}}&\lesssim
    \seminorm{j,h}{\underline{v}_h} \, \tnorm{h}{\underline{v}_h}
    \overset{\eqref{eq:inf-sup:partial.coercivity}}\le
    K \$\, \tnorm{h}{\underline{v}_h}
    + \frac{1}{12} \tnorm{h}{\underline{v}_h}^2,
  \end{aligned}
  \]
  where, in the conclusion, we have additionally used the generalized Young's inequality.
  Finally, a Cauchy--Schwarz inequality readily gives
  \[
  \term_5
  \lesssim \seminorm{j,h}{\underline{v}_h}~\seminorm{j,h}{\underline{z}_h}
  \overset{\eqref{eq:inf-sup:partial.coercivity},\,\eqref{eq:zh:tnorm.estimate}}\lesssim
  \left(
  \$\, \tnorm{h}{\underline{v}_h}
  \right)^{\frac12} \tnorm{h}{\underline{v}_h}
  \overset{\eqref{eq:inf-sup:partial.coercivity}}\le
  K \$\, \tnorm{h}{\underline{v}_h}
  + \frac{1}{12} \tnorm{h}{\underline{v}_h}^2.
  \]
  Gathering the above estimates for $\term_1,\ldots,\term_5$, after rearranging the terms, multiplying by $2$, and noticing that $1 \lesssim \rho$, we end up with
  \begin{equation}\label{eq:inf-sup:frak.T}
    \term
    \le
    K \rho^2 \$\, \tnorm{h}{\underline{v}_h}
    + \frac12 \tnorm{h}{\underline{v}_h}^2.
  \end{equation}
  \underline{(iii) \emph{Conclusion.}}
  The conclusion follows summing~\eqref{eq:inf-sup:frak.T} to~\eqref{eq:inf-sup:partial.coercivity} and simplifying.
\end{proof}

\subsection{Consistency}

\begin{lemma}[Estimate of the consistency error]\label{lem:consistency}
  Under the assumptions of Lemma~\ref{lem:inf-sup}, let $w \in V_0 \cap H^s(\Omega;\Complex^m)$ with $s > \frac12$, set $\huline{w}_h \coloneqq \Ih w$, and define the consistency error
  \begin{equation}\label{eq:Err}
    \Err(w;\underline{v}_h)
    \coloneqq (A w, v_h)_\Omega - a_h(\huline{w}_h, \underline{v}_h).
  \end{equation}
  It holds, denoting by $\tnorm{h,*}{\cdot}$ the norm dual to $\tnorm{h}{\cdot}$ and recalling~\eqref{eq:rho},
  \[
  \tnorm{h,*}{\Err(w;\cdot)}
  \lesssim \left(
  \sum_{T \in \Th} (\Kmax \rho + \tref^{-1}) \norm{L^2(T;\Complex^m)}{w - \widehat{w}_T}^2
  + \sum_{T \in \Th} (\Aref + \Kmin h_T) \norm{L^2(\partial T;\Complex^m)}{w - \widehat{w}_T}^2
  \right)^{\frac12}.
  \]
\end{lemma}

\begin{proof}
  We reformulate the contributions in~\eqref{eq:Err} so as to make the differences $(w - \widehat{w}_Y)_{Y \in \Th \cup \Fh}$ appear.
  Since the quantity $\mathcal{N}_F w$ is single-valued at interfaces (see~\cite[Corollary 60.2]{Ern.Guermond:21*2}), we have
  \[
  0 = \sum_{F \in \Fhb} (\mathcal{N}_F w, v_F)_F
  - \sum_{T \in \Th} \sum_{F \in \FT} \omega_{TF} (w, \mathcal{N}_F v_F)_F,
  \]
  where we have additionally used the fact that the fields $(\mathcal{N}_F)_{F \in \Fh}$ are Hermitian to move $\mathcal{N}_F$ in front of the second argument in the second term.
  Summing~\eqref{eq:ibp:local} with $v = v_T$ over $T \in \Th$ and adding the above expression, we get
  \[
  (A w, v_h)_\Omega
  = \sum_{T \in \Th} (w, \tilde{A} v_T)_T
  - \sum_{T \in \Th} \sum_{F \in \FT} \omega_{TF} (w, \mathcal{N}_F(v_F - v_T))_F
  + \sum_{F \in \Fhb} (\mathcal{N}_F w, v_F)_F.
  \]
  Moreover, since $w \in V_0$, by~\eqref{eq:SFb:stabilization} we have
  \[
  \frac12 \sum_{F \in \Fhb} (\mathcal{M}_F + \mathcal{S}_F^{\rm b} - \mathcal{N}_F) w, v_F)_F = 0.
  \]
  Summing the above relations, we get
  \begin{equation}\label{eq:Eh:reformulation:T1}
    (A w, v_h)_\Omega
    = \sum_{T \in \Th} (w, \tilde{A} v_T)_T
    - \sum_{T \in \Th} \sum_{F \in \FT} \omega_{TF} (w, \mathcal{N}_F(v_F - v_T))_F
    \frac12 \sum_{F \in \Fhb} (\mathcal{M}_F + \mathcal{S}_F^{\rm b} + \mathcal{N}_F) w, v_F)_F.
  \end{equation}
  On the other hand, using~\eqref{eq:ibp:global} to replace the first term in the expression~\eqref{eq:ah} of $a_h$ and letting $\underline{w}_h = \huline{w}_h$, we get
  \begin{equation}\label{eq:Eh:reformulation:T2}
    \begin{aligned}
      a_h(\huline{w}_h, \underline{v}_h)
      &\coloneqq
      \sum_{T \in \Th} (\widehat{w}_T, \tilde{A} v_T)_T
      + \Kmin \sum_{T \in \Th} h_T (\widehat{w}_F - \widehat{w}_T, v_F - v_T)_F
      \\
      &\quad
      - \sum_{T \in \Th} \sum_{F \in \FT} \omega_{TF} \left(
      \frac{\widehat{w}_F + \widehat{w}_T}{2}, \mathcal{N}_F (v_F - v_T)
      \right)_F
      \\
      &\quad
      + \frac12 \sum_{F \in \Fhb} ((\mathcal{M}_F + \mathcal{S}_F^{\rm b} + \mathcal{N}_F) \widehat{w}_F, v_F)_F
      + j_h(\huline{w}_h, \underline{v}_h).
    \end{aligned}
  \end{equation}
  In conclusion, subtracting~\eqref{eq:Eh:reformulation:T2} from~\eqref{eq:Eh:reformulation:T1}, we write
  \[
  \Err(w;\underline{v}_h) = \term_1 + \cdots + \term_5
  \]
  with
  \[
  \begin{aligned}
    \term_1 &\coloneqq
    \sum_{T \in \Th} (w - \widehat{w}_T, \tilde{A} v_T)_T,
    \\
    \term_2 &\coloneqq
    \Kmin \sum_{T \in \Th} h_T \sum_{F \in \FT} (w - \widehat{w}_T, v_F - v_T)_F,
    \\
    \term_3 &\coloneqq
    - \frac12 \sum_{T \in \Th} \sum_{F \in \FT} \omega_{TF} ((w - \widehat{w}_F) + (w - \widehat{w}_T), \mathcal{N}_F (v_F - v_T))_F,
    \\
    \term_4 &\coloneqq
    \frac12 \sum_{F \in \Fhb} ((\mathcal{M}_F + \mathcal{S}_F^{\rm b} + \mathcal{N}_F) (w - \widehat{w}_F), v_F)_F,
    \\
    \term_5 &\coloneqq
    j_h(\huline{w}_h, \underline{v}_h),
  \end{aligned}
  \]
  where we have noticed that $v_F - (v_T)_{|F} \in \Poly{k}(F;\Complex^m)$ and used the definition of $\lproj{k}{F}$ to replace $\widehat{w}_F = \lproj{k}{F} w$ with $w$ in $\term_2$.
  We proceed to estimate the above terms.
  By definition~\eqref{eq:A1.A.tA} of $\tilde{A}$, we have
  \[
  \begin{aligned}
    \term_1
    &= \sum_{T \in \Th} (w - \widehat{w}_T, (\mathcal{K}^\herm - \nabla \cdot \mathcal{A}) v_T)_T
    - \sum_{T \in \Th} (w - \widehat{w}_T, A_1 v_T)_T
    \\
    \overset{\eqref{eq:rho}}&\le
    \sum_{T \in \Th}
    \big(\Kmax \Kmin^{-\frac12} + \tref^{-\frac12}\big) \norm{L^2(T;\Complex^m)}{w - \widehat{w}_T}
    \left(
    \Kmin^{\frac12} \norm{L^2(T;\Complex^m)}{v_T}
    + \tref^{\frac12} \norm{L^2(T;\Complex^m)}{A_1 v_T}
    \right)
    \\
    \overset{\eqref{eq:rho},\,\eqref{eq:tnorm.h},\,\eqref{eq:norm.0.T}}&\lesssim
    \left(
    \sum_{T \in \Th} (\Kmax \rho + \tref^{-1}) \norm{L^2(T;\Complex^m)}{w - \widehat{w}_T}^2
    \right)^{\frac12}
    \tnorm{h}{\underline{v}_h},
  \end{aligned}
  \]
  where we have used $(2,\infty,2)$-H\"{o}lder and Cauchy--Schwarz inequalities on the integrals in the first step and Cauchy--Schwarz inequalities on the sums to conclude.

  Using Cauchy--Schwarz inequalities on the integrals and the sums and recalling~\eqref{eq:tnorm.h}, we have
  \[
  \term_2 \le \left(
  \Kmin \sum_{T \in \Th} h_T \norm{L^2(\partial T;\Complex^m)}{w - \widehat{w}_T}^2
  \right)^{\frac12} \tnorm{h}{\underline{v}_h}.
  \]

  Using~\eqref{eq:STFi:control.NF}, the third term can be bounded as follows:
  \[
  \begin{aligned}
    \term_3
    &\lesssim \sum_{T \in \Th} \sum_{F \in \FT} \Aref^{\frac12} \left(
    \norm{L^2(F;\Complex^m)}{w - \widehat{w}_F}
    + \norm{L^2(F;\Complex^m)}{w - \widehat{w}_T}
    \right)~(\mathcal{S}_{TF}^{\rm i} (v_F - v_T), v_F - v_T)_F^{\frac12}
    \\
    &\lesssim \left(
    \sum_{T \in \Th} \Aref \norm{L^2(\partial T;\Complex^m)}{w - \widehat{w}_T}^2
    \right)^{\frac12}
    \seminorm{j,h}{\underline{v}_h}
    \overset{\eqref{eq:tnorm.h}}\le \left(
    \sum_{T \in \Th} \Aref \norm{L^2(\partial T;\Complex^m)}{w - \widehat{w}_T}^2
    \right)^{\frac12}
    \tnorm{h}{\underline{v}_h}.
  \end{aligned}
  \]
  To pass to the second line, we have used a Cauchy--Schwarz inequality on the sums and noticed that
  \begin{equation}\label{eq:w-hwF:estimate}
    \norm{L^2(F;\Complex^m)}{w - \widehat{w}_F}
    \le \norm{L^2(F;\Complex^m)}{w - \widehat{w}_T}
    + \norm{L^2(F;\Complex^m)}{\lproj{k}{F}(w - \widehat{w}_T)}
    \lesssim \norm{L^2(F;\Complex^m)}{w - \widehat{w}_T},
  \end{equation}
  where we have inserted $\lproj{k}{F} \widehat{w}_T - \widehat{w}_T = 0$ (since $(\widehat{w}_T)_{|F} \in \Poly{k}(F;\Complex^m)$ and $\lproj{k}{F}$ is idempotent) into the norm and used a triangle inequality in the first step,
  while the conclusion follows from the fact that the $L^2$-projector is (trivially) $L^2$-bounded.

  For the fourth term, using Cauchy--Schwarz inequalities followed by~\eqref{eq:SFb:control.NF} and~\eqref{eq:seminorm.S.b}, we readily obtain
  \[
  \term_4 \lesssim \left(
  \sum_{F \in \Fhb} \Aref[T_F] \norm{L^2(F;\Complex^m)}{w - \widehat{w}_{T_F}}^2
  \right)^{\frac12}
  \seminorm{{\rm b},h}{\underline{v}_h}
  \overset{\eqref{eq:tnorm.h}}\le
  \left(
  \sum_{F \in \Fhb} \Aref[T_F] \norm{L^2(F;\Complex^m)}{w - \widehat{w}_{T_F}}^2
  \right)^{\frac12}
  \tnorm{h}{\underline{v}_h}.
  \]

  Finally, for the fifth term, we simply write, using Cauchy--Schwarz inequalities together with~\eqref{eq:STFi:boundedness},
  \[
  \term_5 \lesssim
  \left(
  \sum_{T \in \Th} \Aref \sum_{F \in \FT} \norm{L^2(F;\Complex^m)}{\widehat{w}_F - \widehat{w}_T}^2
  \right)^{\frac12}
  \seminorm{j,h}{\underline{v}_h}
  \le
  \left(
  \sum_{T \in \Th} \Aref \norm{L^2(\partial T;\Complex^m)}{w - \widehat{w}_T}^2
  \right)^{\frac12}
  \tnorm{h}{\underline{v}_h},
  \]
  where, for the first factor we have inserted $\pm w$ into the norm and used a triangle inequality to write
  \[
  \norm{L^2(F;\Complex^m)}{\widehat{w}_F - \widehat{w}_T}
  \le \norm{L^2(F;\Complex^m)}{w - \widehat{w}_F}
  + \norm{L^2(F;\Complex^m)}{w - \widehat{w}_T}
  \overset{\eqref{eq:w-hwF:estimate}}\lesssim
  \norm{L^2(F;\Complex^m)}{w - \widehat{w}_T},
  \]
  while the second factor was estimated using the definition~\eqref{eq:tnorm.h} of $\tnorm{h}{\underline{v}_h}$.
  The conclusion follows using the above bounds for $\term_1,\ldots,\term_4$ in the definition of the dual norm.
\end{proof}

\subsection{Error estimate}

\begin{theorem}[Error estimate]\label{th:err-est}
  Denote by $u \in V_0$ the solution of the weak problem~\eqref{eq:weak}, and further assume that $u \in H^s(\Omega;\Complex^m)$ for some $s > \frac12$.
  Then, under the assumptions of Lemma~\ref{lem:consistency}, it holds,
  \begin{multline}\label{eq:error.estimate}
    \tnorm{h}{\underline{u}_h - \Ih u}
    \\
    \lesssim
    \rho^2 \left(
    \sum_{T \in \Th} (\Kmax \rho + \tref^{-1}) \norm{L^2(T;\Complex^m)}{u - \lproj{k}{T}u}^2
    + \sum_{T \in \Th} (\Aref + \Kmin h_T) \norm{L^2(\partial T;\Complex^m)}{u - \lproj{k}{T}u}^2
    \right)^{\frac12}.
  \end{multline}
  If, additionally, for all $T \in \Th$, $u_{|T} \in H^r(T;\Complex^m)$ for some $r \in \{ 1, \ldots, k\}$,
  \begin{equation}\label{eq:convergence.rate}
    \tnorm{h}{\underline{u}_h - \Ih u}
    \lesssim
    \rho^2 \left[
      \sum_{T \in \Th} \left(
      (\Kmax \rho + \tref^{-1}) h_T^{2(r + 1)} \seminorm{H^r(T;\Complex^m)}{u}^2
      + (\Aref + \Kmin h_T) h_T^{2r + 1} \seminorm{H^r(T;\Complex^m)}{u}^2
      \right)
      \right]^{\frac12}.
  \end{equation}
\end{theorem}

\begin{proof}
  The basic error estimate~\eqref{eq:error.estimate} is an immediate consequence of Lemmas~\ref{lem:inf-sup} and~\ref{lem:consistency} together with the Third Strang Lemma~\cite{Di-Pietro.Droniou:18}.
  The estimate of the convergence rate~\eqref{eq:convergence.rate} follows from the approximation properties of $\lproj{k}{T}$; see~\cite{Di-Pietro.Droniou:17}.
\end{proof}

\begin{remark}[Convergence rates]\label{rem:convergence.rates}
  The right-hand side of~\eqref{eq:convergence.rate} converges asymptotically as $h^{r+\frac12}$.
  Notice, however, that, for coarse enough meshes, one can have $\Aref \le \Kmin h_T$, and therefore observe a pre-asymptotic convergence rate of $h^{r+1}$ for the error.
  From a physical standpoint, this can be interpreted as a dominant-reaction regime.
\end{remark}


\section{Numerical tests}\label{sec:numerical.tests}

This section contains an extensive numerical validation of the proposed method on three-dimensional configurations.
The  implementation is based on the \verb|C++| library \verb|HarDCore|\footnote{https://github.com/jdroniou/HArDCore}.
The sparse solver \verb|BiCGSTAB| of the \verb|Eigen| library is used to solve the statically condended linear systems.
In Sections~\ref{subsec:num-scalar_DAR} and~\ref{subsec:num-vector_DAR}, we numerically assess the convergence rates for the scalar and vector diffusion-advection-reaction equations respectively described in Sections~\ref{subsubsec:scalar-dar} and~\ref{subsubsec:vector-dar}.
In Section~\ref{subsec:magn.field.diff.conv}, we consider a benchmark configuration in magnetohydrodynamics inspired by~\cite{Perry.Jones:78}, corresponding to the expulsion of the magnetic field from a rotating cilinder, and show that the proposed scheme captures this physical phenomenon.

\subsection{Scalar advection-diffusion-reaction}\label{subsec:num-scalar_DAR}

We consider here problem~\eqref{eq:scalar_DAR} with
\[
\Omega = (0,1)^3,\qquad
\Gamma_\dirichlet = \partial \Omega,\qquad
\kappa = \mathcal{I}_3,\qquad
\beta = \begin{bmatrix}1 & 1 & 1
\end{bmatrix}^\trans,\qquad
\mu=1.
\]
The source term $f$ and the Dirichlet boundary datum are set so that the solution corresponds to the scalar potential
\[
p = \sin(\pi x) \sin(\pi y) \sin(\pi z).
\]
The penalty fields ared designed as in Remark \ref{rmk:penalty-fields}.
The test is performed considering two regular mesh families,
respectively composed of simplicial and general polyhedral elements.
The trend of the error $\tnorm{h}{\Ih{u}-\underline{u}_h}$ with respect to the mesh diameter $h$
is recorded for several values of the degree $k$. The results are summarized in Figure \ref{fig:trend_scalar_DAR}.
The observed convergence rates match or exceed the ones resulting from Theorem~\ref{th:err-est}.
Specifically, for the simplicial mesh family, the error decay rate is roughly $k + 1$.
This could be due to the fact that the asymptotic regime is not yet achieved for the finest mesh in our family (see Remark~\ref{rem:convergence.rates}), and possibly also to some super-convergence phenomenon linked to the fact that this mesh is mostly structured.
For the polyhedral mesh family, an order of convergence of roughly $k + \frac12$ is observed for $k \neq 1$, while a higher-than-expected convergence rate of $2.51$ is obtained for the last mesh refinement for $k = 1$.
This could be due both to pre-asymptotic behaviour and possibly a super-convergence phenomenon which will require further investigation.

\begin{figure}
  \begin{subfigure}[t]{\textwidth}\centering
    \begin{tikzpicture}
      \begin{loglogaxis}[height=6.5cm]
\addplot table[x=meshsize,y=H1TypeError] {scalar_DAR_tet-data_rates-0.dat} node[font = {\small}, pos = 0.166666666666667, above=2pt]{1.17}node[font = {\small}, pos = 0.5, above=2pt]{0.76}node[font = {\small}, pos = 0.833333333333333, above=2pt]{1.18};
\addplot table[x=meshsize,y=H1TypeError] {scalar_DAR_tet-data_rates-1.dat} node[font = {\small}, pos = 0.166666666666667, above=2pt]{2.29}node[font = {\small}, pos = 0.5, above=2pt]{1.48}node[font = {\small}, pos = 0.833333333333333, above=2pt]{2.25};
\addplot table[x=meshsize,y=H1TypeError] {scalar_DAR_tet-data_rates-2.dat} node[font = {\small}, pos = 0.166666666666667, above=2pt]{2.69}node[font = {\small}, pos = 0.5, above=2pt]{2.45}node[font = {\small}, pos = 0.833333333333333, above=2pt]{3.13};
\addplot table[x=meshsize,y=H1TypeError] {scalar_DAR_tet-data_rates-3.dat} node[font = {\small}, pos = 0.166666666666667, above=2pt]{3.93}node[font = {\small}, pos = 0.5, above=2pt]{3.06}node[font = {\small}, pos = 0.833333333333333, above=2pt]{4.06};
\logLogSlopeTriangle{0.90}{0.4}{0.1}{0}{black};
\logLogSlopeTriangle{0.90}{0.4}{0.1}{1.5}{black};
\logLogSlopeTriangle{0.90}{0.4}{0.1}{2.5}{black};
\logLogSlopeTriangle{0.90}{0.4}{0.1}{3.5}{black};
      \end{loglogaxis}
    \end{tikzpicture}
    \hspace{1cm}
    \includegraphics[height=5.5cm]{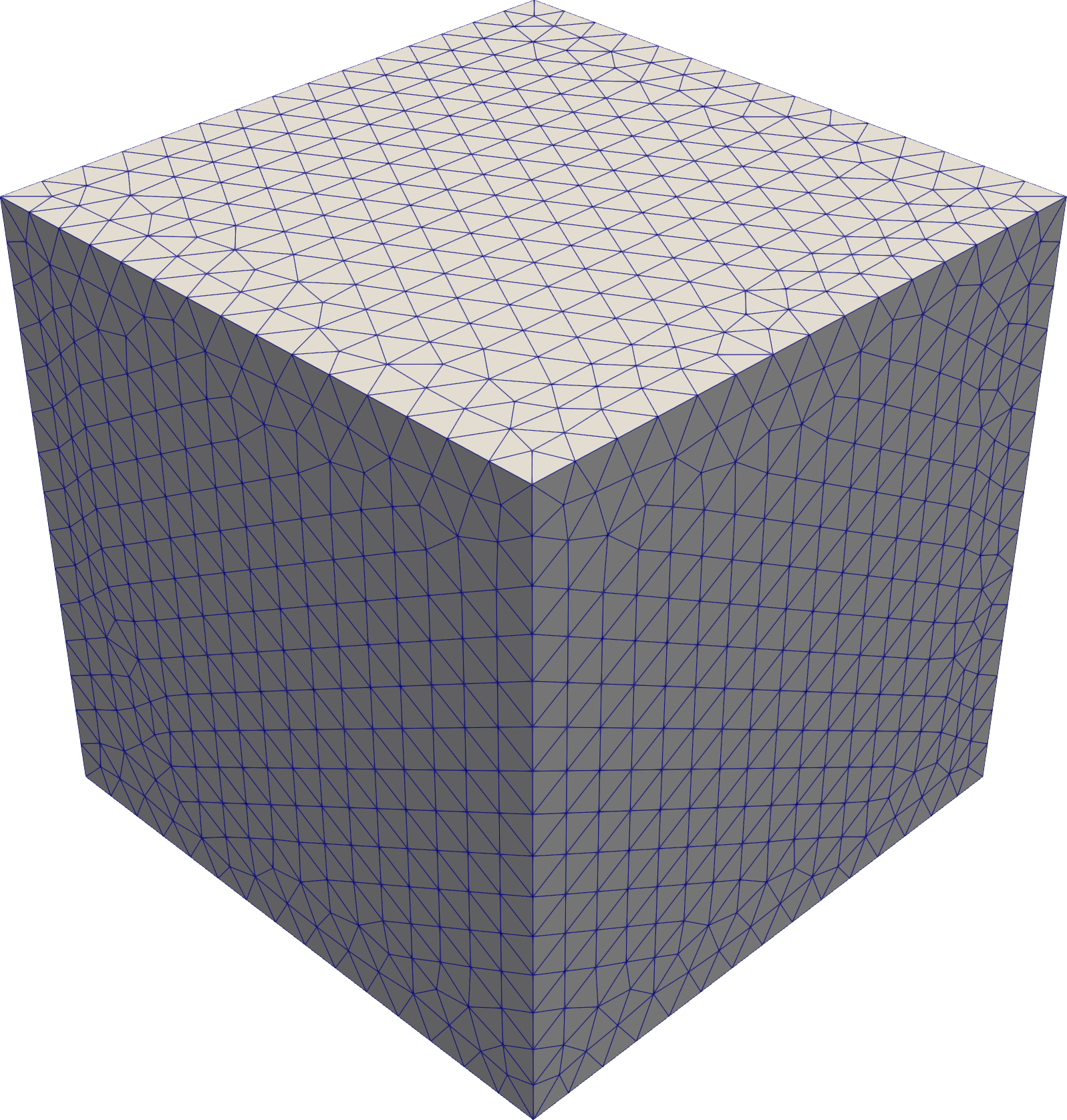}
    \caption{Simplicial mesh family}
  \end{subfigure}
  \medskip\\
  \begin{subfigure}[t]{\textwidth}\centering
    \begin{tikzpicture}
      \begin{loglogaxis}[height=6.5cm]
\addplot table[x=meshsize,y=H1TypeError] {scalar_DAR_voro-data_rates-0.dat} node[font = {\small}, pos = 0.166666666666667, above=2pt]{0.36}node[font = {\small}, pos = 0.5, above=2pt]{0.71}node[font = {\small}, pos = 0.833333333333333, above=2pt]{0.60};
\addplot table[x=meshsize,y=H1TypeError] {scalar_DAR_voro-data_rates-1.dat} node[font = {\small}, pos = 0.166666666666667, above=2pt]{1.63}node[font = {\small}, pos = 0.5, above=2pt]{1.96}node[font = {\small}, pos = 0.833333333333333, above=2pt]{2.51};
\addplot table[x=meshsize,y=H1TypeError] {scalar_DAR_voro-data_rates-2.dat} node[font = {\small}, pos = 0.166666666666667, above=2pt]{3.27}node[font = {\small}, pos = 0.5, above=2pt]{2.27}node[font = {\small}, pos = 0.833333333333333, above=2pt]{2.80};
\addplot table[x=meshsize,y=H1TypeError] {scalar_DAR_voro-data_rates-3.dat} node[font = {\small}, pos = 0.166666666666667, above=2pt]{3.68}node[font = {\small}, pos = 0.5, above=2pt]{4.31}node[font = {\small}, pos = 0.833333333333333, above=2pt]{3.62};
\logLogSlopeTriangle{0.90}{0.4}{0.1}{0}{black};
\logLogSlopeTriangle{0.90}{0.4}{0.1}{1.5}{black};
\logLogSlopeTriangle{0.90}{0.4}{0.1}{2.5}{black};
\logLogSlopeTriangle{0.90}{0.4}{0.1}{3.5}{black};
      \end{loglogaxis}
    \end{tikzpicture}
    \hspace{1cm}
    \includegraphics[height=5.5cm]{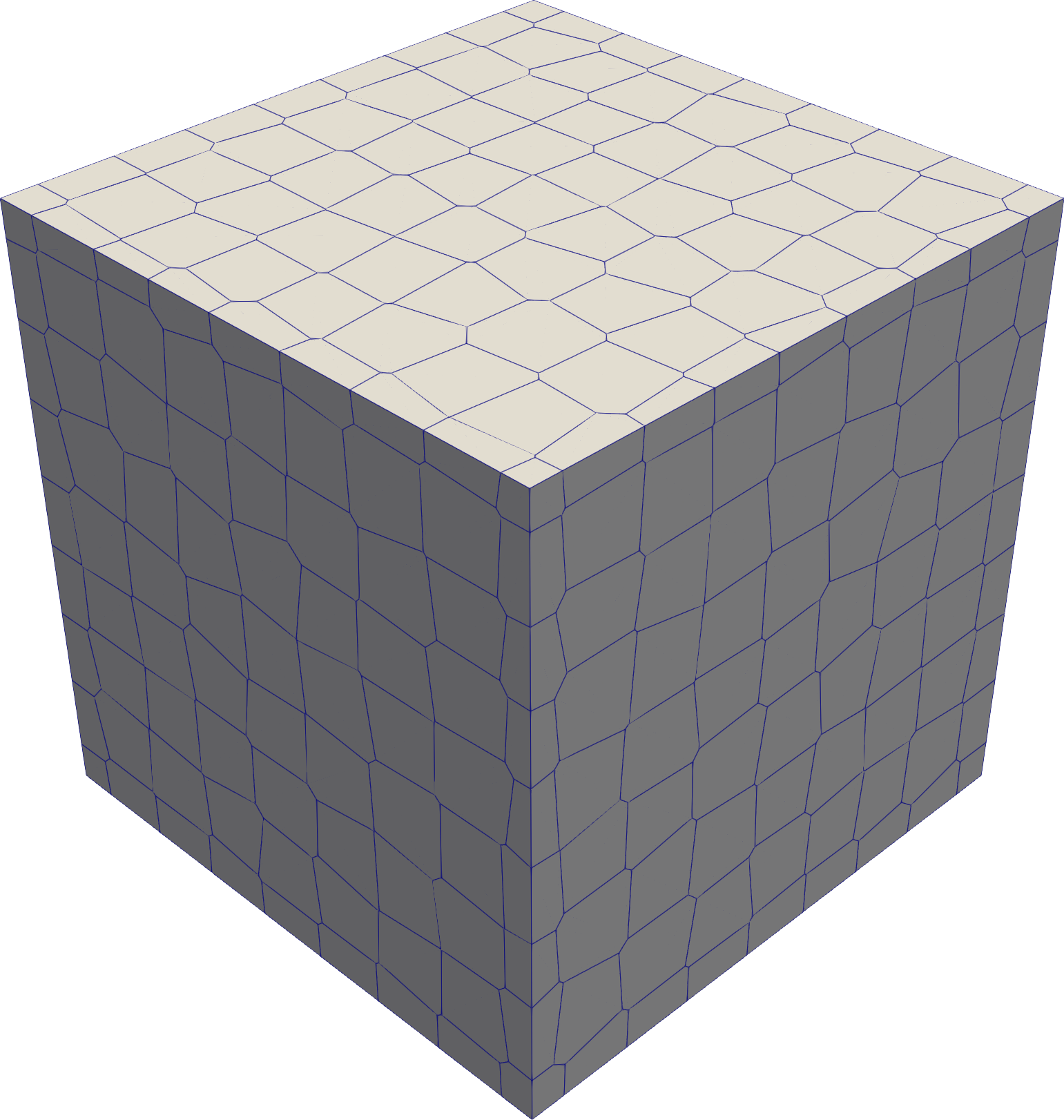}
    \caption{Polyhedral mesh family}
  \end{subfigure}
  \caption{Approximation error $\tnorm{h}{\Ih{u}-\underline{u}_h}$ v. meshsize $h$ for the test described in Section~\ref{subsec:num-scalar_DAR}. Convergence slopes are displayed for two different mesh families.
    Specimens of the two mesh families with the smallest diameter
    are represented on the right.}\label{fig:trend_scalar_DAR}
\end{figure}
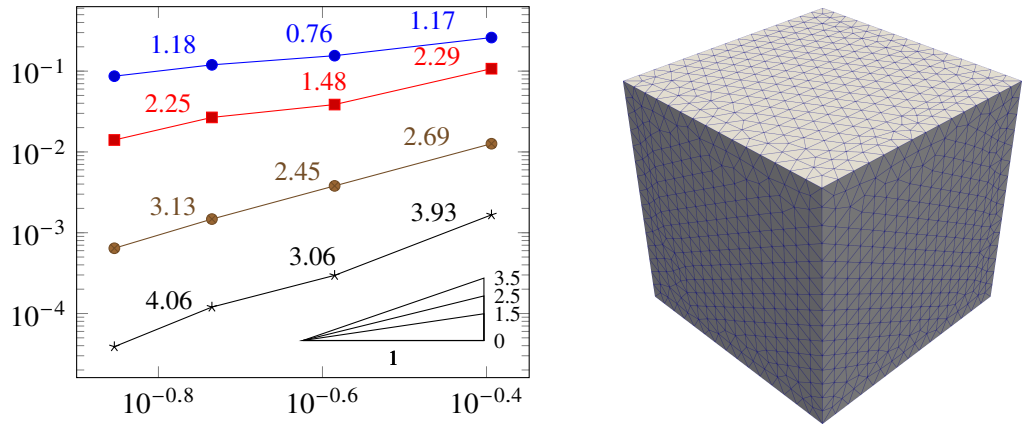
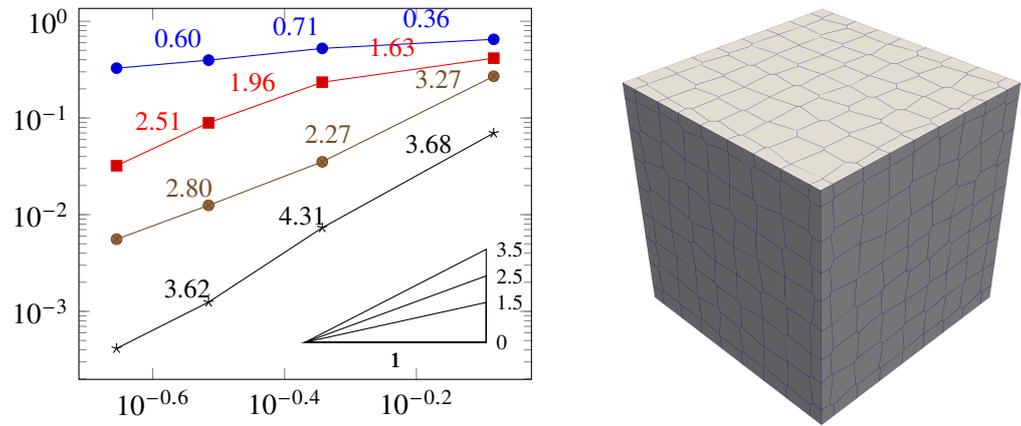

\subsection{Vector advection-diffusion-reaction}\label{subsec:num-vector_DAR}

We consider problem~\eqref{eq:vector_DAR} with
\[
\Omega = (0,1)^3,\qquad
\Gamma_\dirichlet = \partial \Omega,\qquad
\varepsilon = 1,\qquad
\beta = \begin{bmatrix}1 & 1 & 1
\end{bmatrix}^\trans,\qquad
\gamma=1.
\]
The source term $f$ and the Dirichlet boundary datum are set so that the solution corresponds to the following vector potential.
\[
p = \begin{bmatrix}
  \sin(\pi z) &
  \sin(\pi x) &
  \sin(\pi y)
\end{bmatrix}^\trans.
\]
The penalty fields ared designed as in Remark~\ref{rmk:penalty-fields}.
The test is performed considering the same refined mesh families as in the previous points, respectively composed of simplicial and general polyhedral elements.
The trend of the error $\tnorm{h}{\Ih{u}-\underline{u}_h}$ with respect to the mesh size $h$ is recorded for several values of the degree $k$.
The results are summarized in Figure~\ref{fig:trend_vector_DAR}.
As for the scalar case, the observed convergence rates match or exceed the predicted ones, with superconvergence in $h^{k+1}$ for the simplicial mesh family.
The explanation of this phenomenon is likely the same as for the scalar case.

\begin{figure}
  \begin{subfigure}[t]{0.49\textwidth}\centering
    \begin{tikzpicture}
      \begin{loglogaxis}[width=\linewidth]
\addplot table[x=meshsize,y=H1TypeError] {vector_DAR_tet-data_rates-0.dat} node[font = {\small}, pos = 0.166666666666667, above=2pt]{0.96}node[font = {\small}, pos = 0.5, above=2pt]{1.03}node[font = {\small}, pos = 0.833333333333333, above=2pt]{1.00};
\addplot table[x=meshsize,y=H1TypeError] {vector_DAR_tet-data_rates-1.dat} node[font = {\small}, pos = 0.166666666666667, above=2pt]{2.18}node[font = {\small}, pos = 0.5, above=2pt]{1.68}node[font = {\small}, pos = 0.833333333333333, above=2pt]{2.16};
\addplot table[x=meshsize,y=H1TypeError] {vector_DAR_tet-data_rates-2.dat} node[font = {\small}, pos = 0.166666666666667, above=2pt]{2.98}node[font = {\small}, pos = 0.5, above=2pt]{2.81}node[font = {\small}, pos = 0.833333333333333, above=2pt]{3.04};
\addplot table[x=meshsize,y=H1TypeError] {vector_DAR_tet-data_rates-3.dat} node[font = {\small}, pos = 0.25, above=2pt]{4.34}node[font = {\small}, pos = 0.75, above=2pt]{3.70};
\logLogSlopeTriangle{0.90}{0.4}{0.1}{0}{black};
\logLogSlopeTriangle{0.90}{0.4}{0.1}{1.5}{black};
\logLogSlopeTriangle{0.90}{0.4}{0.1}{3.5}{black};
      \end{loglogaxis}
    \end{tikzpicture}
    \caption{Simplicial mesh family}
  \end{subfigure}
  \begin{subfigure}[t]{0.49\textwidth}\centering
    \begin{tikzpicture}
      \begin{loglogaxis}[width=\linewidth]
\addplot table[x=meshsize,y=H1TypeError] {vector_DAR_voro-data_rates-0.dat} node[font = {\small}, pos = 0.166666666666667, above=2pt]{0.43}node[font = {\small}, pos = 0.5, above=2pt]{0.46}node[font = {\small}, pos = 0.833333333333333, above=2pt]{0.44};
\addplot table[x=meshsize,y=H1TypeError] {vector_DAR_voro-data_rates-1.dat} node[font = {\small}, pos = 0.166666666666667, above=2pt]{1.33}node[font = {\small}, pos = 0.5, above=2pt]{1.83}node[font = {\small}, pos = 0.833333333333333, above=2pt]{1.93};
\addplot table[x=meshsize,y=H1TypeError] {vector_DAR_voro-data_rates-2.dat} node[font = {\small}, pos = 0.166666666666667, above=2pt]{2.77}node[font = {\small}, pos = 0.5, above=2pt]{2.34}node[font = {\small}, pos = 0.833333333333333, above=2pt]{2.85};
\addplot table[x=meshsize,y=H1TypeError] {vector_DAR_voro-data_rates-3.dat} node[font = {\small}, pos = 0.166666666666667, above=2pt]{3.00}node[font = {\small}, pos = 0.5, above=2pt]{4.06}node[font = {\small}, pos = 0.833333333333333, above=2pt]{3.78};
\logLogSlopeTriangle{0.90}{0.4}{0.1}{0}{black};
\logLogSlopeTriangle{0.90}{0.4}{0.1}{1.5}{black};
\logLogSlopeTriangle{0.90}{0.4}{0.1}{2.5}{black}

      \end{loglogaxis}
    \end{tikzpicture}
    \caption{Polyhedral mesh family}
  \end{subfigure}
  \caption{Approximation error $\tnorm{h}{\Ih{u}-\underline{u}_h}$ v. meshsize $h$ for the test described in Section~\ref{subsec:num-vector_DAR}. Convergence slopes are displayed for the mesh families depicted in Figure~\ref{fig:trend_scalar_DAR}.}\label{fig:trend_vector_DAR}
\end{figure}
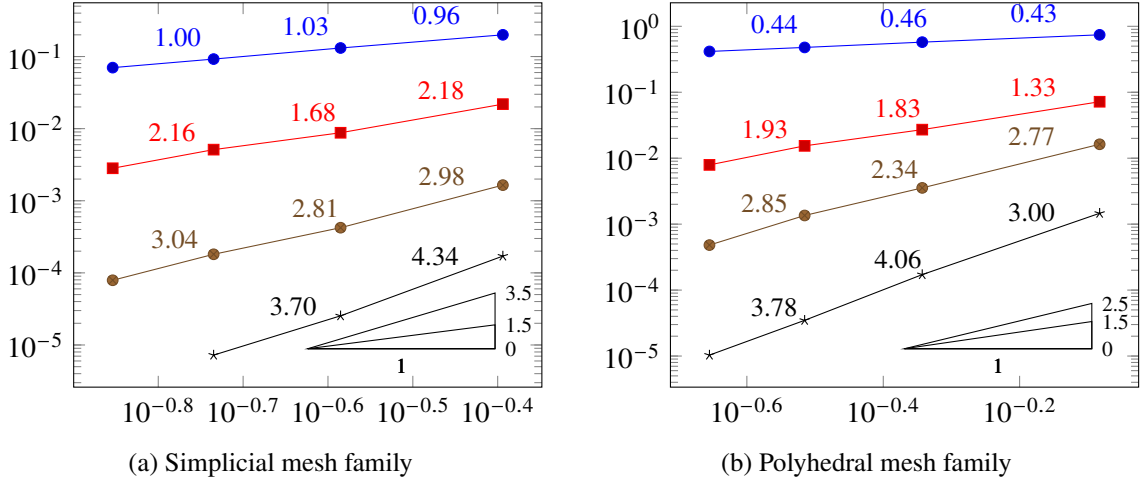

\subsection{Diffusion-advection of the magnetic field in a rotating cylinder}\label{subsec:magn.field.diff.conv}

In the context of magnetohydrodynamics, the stationary magnetic field $B$ in a medium with uniform conductivity $\sigma > 0$ and magnetic permeability $\mu > 0$ advected by a velocity field $\beta$ such that $\nabla \cdot \beta = 0$ can be described by the (steady) induction equation:
\begin{equation}\label{eq:ind.eq}
  \nabla \times (\beta \times B)  + \frac{1}{\sigma \mu} \nabla^2 B = 0.
\end{equation}
For a derivation of~\eqref{eq:ind.eq} starting from Maxwell equations, we refer to~\cite[Chapter 1, Section II]{Moreau:13}.
Assuming sufficient regularity for the following manipulations to make sense, we recast the induction equation into the framework described in Section~\ref{subsubsec:vector-dar} using the identities:
\[
\begin{aligned}
  \nabla^2 B &= - \nabla \times (\nabla \times B) + \nabla (\nabla \cdot B),\\
  \nabla \times (\beta \times B) &= (B\cdot\nabla)\beta - (\beta\cdot\nabla) + (\nabla\cdot B) \beta - (\nabla\cdot \beta)B.
\end{aligned}
\]
Additionally recalling that $\nabla\cdot B = \nabla \cdot \beta = 0$, we get
\begin{equation}\label{eq:ind.eq.reform}
  \nabla\times\left(\frac{1}{\sigma\mu} \nabla\times B\right) + (\beta\cdot\nabla)B - (B\cdot\nabla)\beta = 0.
\end{equation}
Letting
\[
p = B, \qquad \varepsilon = \frac{1}{\sigma\mu},\qquad \gamma=0, \qquad f_p = 0,
\]
\eqref{eq:ind.eq.reform} can be cast in the setting~\eqref{eq:vector_DAR.as.fried} by modifying the field $\mathcal{K}$ as follows:
\[
\mathcal{K} \coloneqq
\begin{bmatrix}
  \varepsilon^{-1} \Id{3} & \mathcal{O}_{3,3} \\
  \mathcal{O}_{3,3} & -\nabla \beta
\end{bmatrix}.
\]
In this context, $b=\sigma^{-1} J$, where $J$ has the meaning of current density.

An experimental configuration for which the magnetic field can be described in terms of the induction equation~\eqref{eq:ind.eq} is presented in~\cite{Perry.Jones:78}. In this work, the authors consider an infinitely long cylindrical rod of conducting material surrounded by void.
The rod is placed in a region with a uniform magnetic field parallel to te section of the cylinder and is put in circular motion around its axis with a constant angular velocity.
The magnetic field tends to be expelled from the rotating region (see~\cite[Figure 2]{Perry.Jones:78} and also, for a similar configuration,~\cite[Figure II.7]{Moreau:13}).
The expulsion of the magnetic field outside the circulating region is more pronounced for higher values of magnetic Reynolds number $\Rm \coloneqq \sigma\mu \norm{L^\infty(\Omega)^d}{\beta} L$, ($L$ being a characteristic length of the system), corresponding to advection-dominated regimes.

To simulate the experimental setting, we consider a cubic domain $\Omega=\left(-\frac{L}{2},\frac{L}{2}\right)^3$ centered in the origin.
The axis of the cylinder is parallel to the $z$ axis and passes through the center of the domain.
Let $R$ be the radius of the cylinder, $\omega$ be the angular velocity and $k$ a user-dependent parameter.
Introducing cylindrical coordinates $(r,\theta,z)$ such that $x=r \cos \theta$, $y=r\sin\theta$, $z=z$ and the associated covariant basis $(\hat{e}_r, \hat{e}_\theta, \hat{e}_z)$, the velocity field is defined as
\[
\beta = \omega r \phi(r) \hat{e}_\theta, \qquad \phi(r) = \frac{1}{1+ e^{\lambda(r-R)}}.
\]
The sigmoid function $\phi$ is a smooth approximation of a step function such that $\phi \approx 1$ for $r<R$ and $\phi(r)\approx 0$ for $r>R$.
When $\lambda$ is chosen large enough, the velocity field reproduces the rigid rotation of the cylinder immersed in a void region.
The problem is solved with $\Gamma_D=\partial\Omega$, and the boundary data is set so that the solution matches the uniform field $B_0 = \hat{e}_x$ along the boundary. The numerical tests are performed setting $\mu\sigma=1$ and varying the magnitude of $\Rm$ through $\omega$.
The results displayed in Figure \ref{fig:expulsion-b-field} capture the expected behaviour: as the magnetic Reynolds number increases from $0$ (unperturbed configuration) to $5 \cdot 10^{-1}$, one observes that the fields line tend to avoid the region occupied by the cilinder.
This effect becomes more and more visible as $\Rm$ increases.

\begin{figure}\centering
  \begin{subfigure}{0.40\textwidth}
    \centering
    \includegraphics[width=\textwidth]{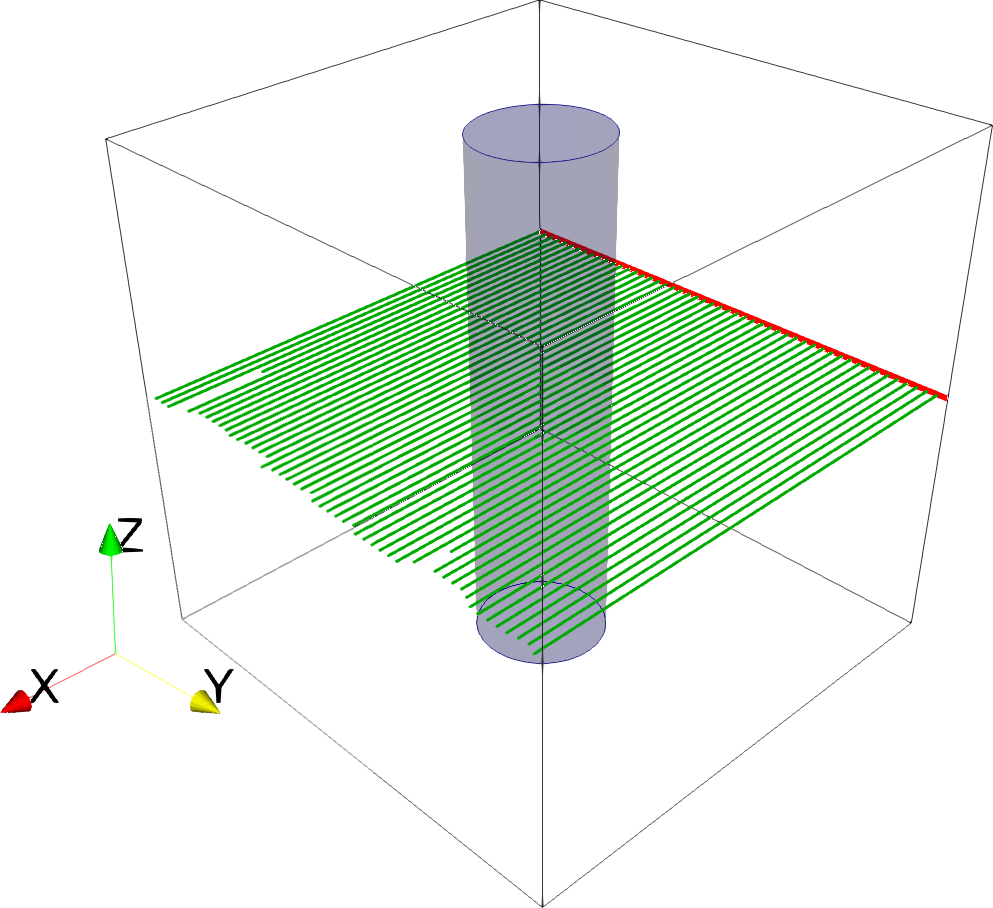}
    \caption{$\Rm = 0$}
  \end{subfigure}
  \begin{subfigure}{0.40\textwidth}
    \centering
    \includegraphics[width=\textwidth]{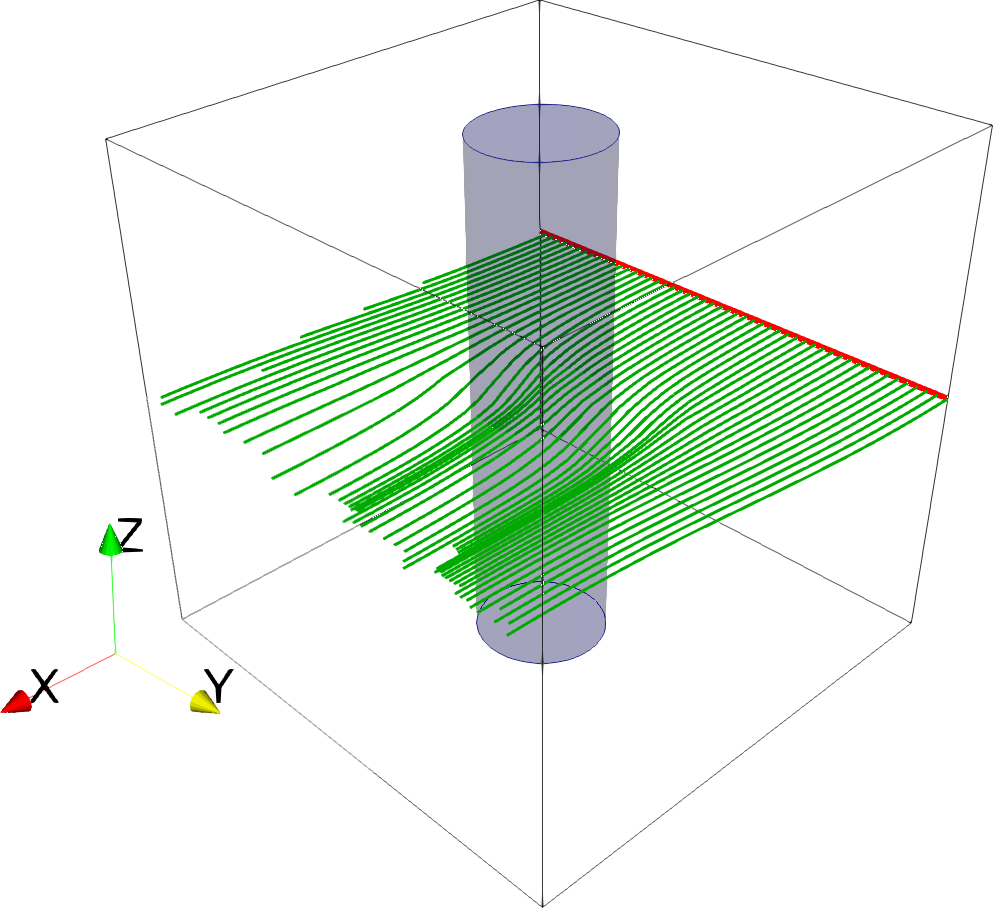}
    \caption{$\Rm = 10^{-1}$}
  \end{subfigure}
  \\
  \begin{subfigure}{0.40\textwidth}
    \centering
    \includegraphics[width=\textwidth]{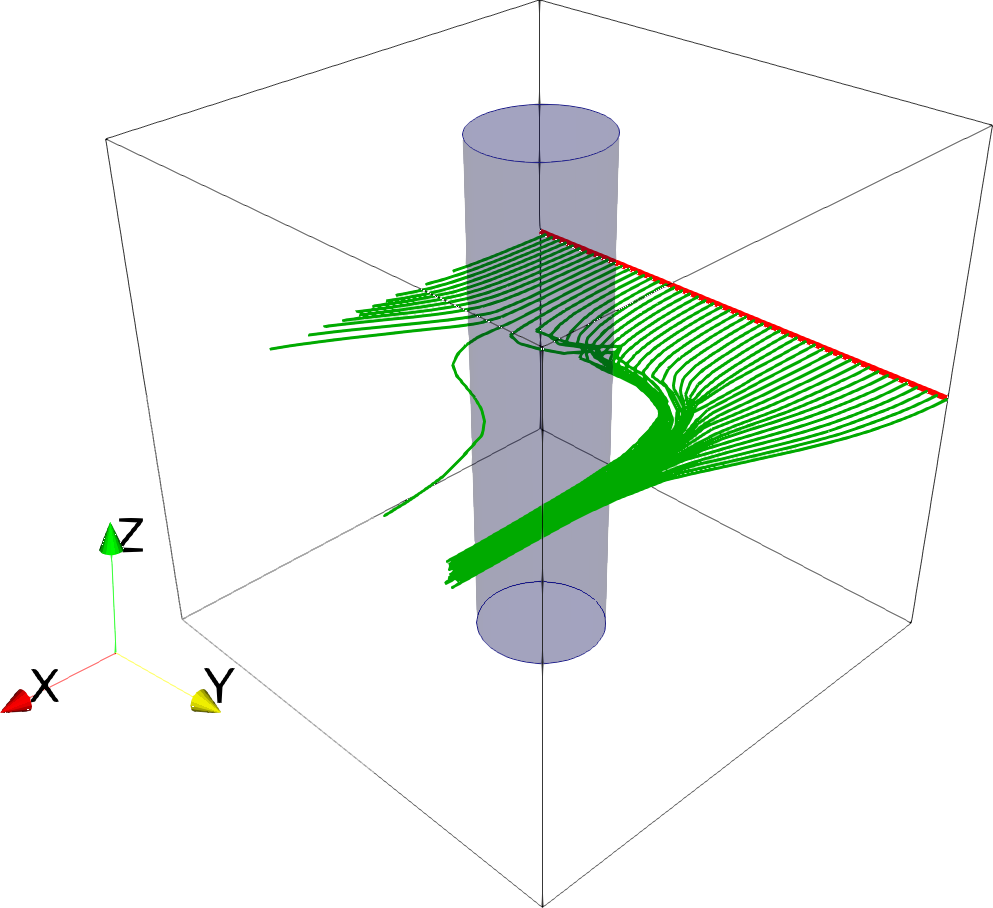}
    \caption{$\Rm = 5\cdot10^{-1}$}
  \end{subfigure}
  \caption{Numerical results for the test described in Section~\ref{subsec:magn.field.diff.conv}.
    The profile of the rotating cylinder is displayed in blue at the center. The numerical solution is shown
    by displaying the field lines of the magnetic field $B$ stemming from points distributed along the line marked in red.
    When the cylinder is still ($\Rm=0$), the magnetic field is equal to the uniform external field $B_0$.
    When the cylinder moves, the line fields are pushed outside of the cylinder. The expulsion effect is larger
    for larger values of $\Rm$.
  }\label{fig:expulsion-b-field}

\end{figure}


\section*{Acknowledgements}

Funded by the European Union (ERC Synergy, NEMESIS, project number 101115663).
Views and opinions expressed are however those of the authors only and do not necessarily reflect those of the European Union or the European Research Council Executive Agency. Neither the European Union nor the granting authority can be held responsible for them.


\printbibliography

\end{document}